\documentclass[11pt,a4paper,reqno]{amsart}
\usepackage{xcolor}
\usepackage{color}
\usepackage{dsfont}
\usepackage{bbold}
\usepackage{latexsym}
\usepackage{amssymb}
\usepackage{enumitem}
\usepackage{array}

\usepackage{amsmath, amsth  m}
\usepackage{amsfonts}
\numberwithin{equation}{section}

\newcommand{\be}{\begin{equation}}
\newcommand{\ee}{\end{equation}}
\newcommand{\ba}{\begin{array}}
\newcommand{\ea}{\end{array}}
\newcommand{\bea}{\begin{eqnarray}}
\newcommand{\eea}{\end{eqnarray}}
\newcommand{\bee}{\begin{eqnarray*}}
\newcommand{\eee}{\end{eqnarray*}}

\newtheorem{Thm}{Theorem}
\newtheorem{Lemma}{Lemma}
\newtheorem{Prop}{Proposition}
\newtheorem{Definition}{Definition}
\newtheorem{Cor}{Corollary}
\newtheorem{Rk}{Remark}
\newcommand{\appendice}[1]
{
  \renewcommand{\thesection}{\fnsymbol{section}}
  \renewcommand{\thesection}{\arabic{section}}
}
\def\ds{\displaystyle}

\def\RR{{\mathbb R}}
\def\NN{{\mathbb N}}

\def\lim{\mathop{\rm lim}}

\def\sup{\mathop{\rm sup}}

\def\e{\varepsilon}

\def\log{{\rm log}}

\def\pa{\partial}

\def\fref#1{{\rm (\ref{#1})}}

\title[Hardy-Littlewood type stability  inequalities] {Extended rearrangement inequalities   and applications to some quantitative stability results}
\author{M. Lemou}
\address{CNRS and IRMAR
University of Rennes 1, Campus de Beaulieu,
Rennes, France.}
\date{}

\begin{document}
\maketitle
\begin{abstract}
In this paper, we prove a new functional inequality of  Hardy-Littlewood type for generalized rearrangements of functions.  We then show how this inequality provides  {\em quantitative} stability results of steady states to evolution systems that essentially preserve the rearrangements and some suitable energy functional, under minimal regularity assumptions on the perturbations.
In particular, this inequality yields a {\em quantitative} stability result of a large class of steady state solutions to the Vlasov-Poisson systems, and more precisely we derive a quantitative  control of the $L^1$ norm of the perturbation by the relative Hamiltonian (the energy functional) and  rearrangements. A general non linear stability  result 
 has been  obtained in \cite{LMR} in the gravitational context, however the  proof relied in a crucial way on compactness arguments which by construction provides no quantitative
 control of the perturbation. Our functional inequality  is also applied   to the context of 2D-Euler system and also provides quantitative stability results of a large class of steady-states  to this system in a natural energy space.  
\end{abstract}


\section{Motivations and main results}

The goal of this work is to prove a new functional inequality for generalized rearrangements which may be applied to get {\em quantitative stability
results} for  evolution systems that  essentially preserve the  equimeasurability (or the symmetric rearrangements) and a suitable energy functional (Hamiltonian, angular momentum, etc).  Examples of such evolution systems are the Vlasov-Poisson and 2D Euler models, on which the application of our inequality will be focused. We first consider a class of steady states $f_0$ to  the Vlasov-Poisson system, which are decreasing functions of their microscopic energy,  and establish a  functional inequality  that gives an explicit control of the $L^1$ distance between $f_0$ and any  function $f$ in terms of  a relative energy functional ${\mathcal H} (f)- {\mathcal H} (f_0)$ and the  $L^1$ distance between the two rearrangements $f^*$ and $f_0^*$ of $f$ and $f_0$.  
An immediate consequence is the nonlinear stability of such non homogeneous steady states  $f_0$.  This stability result was  already established in \cite{LMR} but  a main part of the proof  is based on compactness arguments which, by construction, provides no {\em quantitative control} of the perturbation.  
As a consequence of our functional inequality  a complete quantitative stability result for models like gravitational Vlasov-Poisson systems is therefore obtained.  Note that in the particular case of homogeneous steady states to Vlasov-Poisson for plasmas,  a quantitative control of equimeasurable perturbations has already been obtained
in \cite{MP2},  and this result has been used in \cite{cag-rou} to prove that the $N$ particle approximation of these steady states is uniformly valid on a time scale much larger that the usual $\log N$ scale.

We then focus on the application to the 2D Euler system  and first revisit the stability results that have been obtained in \cite{MP1} for some particular classes of  steady states: planer couette flow and radial steady states.  We show that these results can be easily derived as an application of our functional inequality under minimal regularity assumptions. Moreover, our proof allows to  extend the results in \cite{MP1} to situations where the domain is not bounded and also to situations where the perturbations are not necessarily  compactly supported or in $L^\infty$.
 The stability of more general steady states to 2D Euler system have been proved in  \cite{Burt1}  on the basis of variational approaches and compactness arguments. More precisely, it is proven in \cite{Burt1} that all strict local maximizers and minimizers of the energy are non linearly stable in  $L^p$, $p>4/3$. In the present work, we show that our functional inequality can also be applied to get a quantitative stability result for steady states which are decreasing function of the stream function (i.e. minimizers of the energy). While the stability of the whole class of steady states which are increasing in terms of the stream function is still an open question, we believe that our strategy may be applied to get a quantitative stability of the sub-class of this set of steady sates which are strict local maximizers of the energy.  This last problem is currently under study and is beyond the scope of the present paper.
 
We finally emphasize that asymptotic stability results have been obtained in the recent past 
 for a class of homogenous  steady states to Vlasov-Poisson (Landau damping) in the seminal work  \cite{MV1},  with strong regularity assumptions.  Similar important results have been obtained in \cite{Bed.Mas} for planer couette flow steady  states solutions to the incompressible 2D Euler system. 
Here we are interested in a different class of steady states including: non homogeneous  steady states for VP, radial steady state solutions  to 2D Euler, and  steady sate solutions to 2D Euler which are decreasing profiles of the stream function. Note that the possible Landau damping phenomena around these steady states is a challenging open question even at the linear level. Here, we only prove the non linear stability of these steady states in a quantitative way, under minimal regularity assumptions. 



The paper is organized as follows. In subsection 1.1  we introduce basic notations, define a generalized notion of rearrangements, and state  our refined Hardy-Littlewood type inequality. In subsections 1.2  and  1.3 we investigate the applications to Vlasov-Poisson and 2D-Euler systems, and state the corresponding
quantitative stability inequalities. In section 2, we give the proofs of the statements of subsection 1.1 and in particular we prove the extended  Hardy-Littlewood functional inequalities. Finally, in section 3, we give the proofs of the stability inequalities (stated in   subsections 1.2  and  1.3) for both Valsov-Poisson and 2D Euler systems.


\subsection{Definitions and main results}

Let $\Omega$ be a measurable domain of $\RR^d$, $d\geq 1$, which is not necessarily bounded. For any nonnegative function $q\in L^1(\Omega)$,  we  define the  associated distribution function 
\be
\label{mu}
\mu_q(t) = \mbox{meas}\{ x\in \Omega,   q(x) >t\}, \quad \forall t\geq 0.
\ee
In particular $\mu_q(t)=0$ if $t\geq \|q\|_{L^{\infty}}$.  Here, and in all the sequel, $\mbox{meas} (A)$ denotes the usual Lebesgue measure on $\RR^d$ of a measurable set $A$, and $\mbox{Supp}(q)$ will denote the support of a function $q\in L^1$. Note that $\mu_q(0)\leq \mbox{meas}(\mbox{Supp}(q))$ 
may be infinite but $\mu_q(t)$ is finite for all $t>0$ since we have
\be
\label{mufinite}
\mu_q(t) \leq \mbox{min}\left(\mu_q(0),\frac{1}{t} \|q\|_{L^1(\Omega)}\right) \qquad \forall t>0.
\ee
By analogy with the usual Schwarz rearrangement on $\RR^d$, we define the rearrangement  $q^*$  of $q$ on $\Omega$  as follows. Let $q^{\sharp}$ be the pseudo-inverse of the function $\mu_q$ defined by
\be
\label{def-qdiez}
q^{\sharp}(s)= \mbox{inf}\{t\geq 0, \mu_q(t)\leq s\}=\mbox{sup}\{t\geq 0, \mu_q(t)> s\} ,\quad \forall s\geq 0,
\ee  
with, in particular, $ q^{\sharp}(0)= \|q\|_{L^{\infty}} \in \RR \cup \{+\infty\}$ and $ q^{\sharp}(+\infty)=0.$
We then define the rearrangement  $q^*$ of $q$ on $\Omega$   by the formula
\be
\label{def-q*}
q^*(x) = q^{\sharp}\left(\mbox{meas}(B_d(0,|x|) \cap \Omega)\right), \quad \forall \ x \in \Omega,
\ee 
where $B_d(0,|x|)$ is the ball (in $\RR^d$) centered at $0$ with radius $|x|$.  It is well known that $q$ and $q^*$ are  equimeasurable, which means that 
$ \mu_q (\lambda)= \mu_{q^*}(\lambda), \forall \lambda\geq 0,$ and  in particular
$$ \|q\|_{L^p(\Omega)} = \|q^*\|_{L^p(\Omega)}, \quad  \mbox{for all} \quad 1\leq p\leq \infty, $$
as soon as such quantities exist. The equimeasurability property of $q$ an $q^*$ is a  particular consequence of a more general statement in Proposition \ref{f*theta-equi} which will be proved later on.

We now introduce an extension of this notion of rearrangement which will be an important tool in this work.
\begin{Definition}
\label{def1}
Let  $\sigma$ be a measurable function on  a domain $\Omega \subset \RR^d$, $d\geq 1$, and let  $a_\sigma$ be  the following associated  nondecreasing   function 
\be
\label{atheta}
\forall e\in \RR, \quad a_{\sigma}(e) = \mbox{meas}\{ x\in \Omega,   \sigma(x) < e\} \in [0,\mbox{meas} (\Omega)].
\ee 
Assume that  the  set $\{e\in \RR:  a_{\sigma}(e) < \mbox{meas} (\Omega)\}$ is not empty and that
\be
\label{A2} e_{min} := \mbox{\em ess inf} \  \sigma <  e_{max}:= \sup\{e\in \RR;  a_{\sigma}(e) < \mbox{meas} (\Omega)\}.
\ee
Assume further that
 \be
\label{levelset0}  \quad  \mbox{meas}\{x\in \Omega;  \sigma(x)=e\}=0, \quad \forall e\in ]-\infty, e_{max}[,
\ee 
and that  
$\lim\limits_{\scriptsize{\begin{array}{c}e\to e_{max}\\ e< e_{max}\end{array}}} a_\sigma (e)= \mbox{meas} (\Omega).$

Then we define the $\sigma$-rearrangement  $q^{*\sigma}$ of  a nonnegative function $q\in L^1(\Omega)$  by the formula
\be
\label{def-q*theta}
q^{*\sigma}(x) = q^{\sharp}(a_\sigma(\sigma(x)))\mathds{1}_{\sigma(x)<e_{max}}, \quad \forall x\in \Omega,
\ee
where $q^{\sharp}$ is defined by \fref{def-qdiez}, with $ q^{\sharp}(0)= \|q\|_{L^{\infty}} \in \RR \cup \{+\infty\}$ and $ q^{\sharp}(+\infty)=0.$
\end{Definition}
\begin{Rk}
\label{Rk1} The domain $\Omega$ is not necessarily of finite measure, and our definition will be used for domains with finite measure (bounded domains for instance) and for domains of infinite measure ($\RR^d$ for instance) as well.  Note  that we may have  $e_{max}=+\infty$ and/or $e_{min}= -\infty$.
\end{Rk}
Note that  no integrability  assumption is assumed on $\sigma$,  in particular $\sigma$ is not necessarily in some $L^p$. In fact Definition \ref{def1} of $q^{*\sigma}$ provides an extension of  the usual Schwarz symmetrization  which  corresponds to the particular case $\Omega = \RR^d$ and $\sigma(x)=|x|$;  we have $q^* = q^{*|x|} $.   It also extends  the rearrangement introduced in \cite{LMR}, which is a rearrangement with respect to a specific quantity (the microscopic energy), to analyze the stability problem of steady state solutions to the Vlasov-Poisson equation. Slightly weaker  assumptions may probably be made on $\sigma$ to define our $\sigma-$rearrangement \fref{def-q*theta}, but we prefer to keep our original assumptions in order to simplify the statement, and as we shall see, the framework of definition \ref{def1} will be sufficient for our applications.  
This new $\sigma-$rearrangement has the following important property which will be proved  in section \ref{proof-prop1}.
\begin{Prop}
\label{f*theta-equi}
Let $\sigma $ be a measurable function on a domain $\Omega \subset \RR^d$ satisfying the assumptions of  Definition \ref{def1}. Then for any nonnegative function $f\in L^1(\Omega)$, the function $f^{*\sigma}$ defined on $\Omega$ by 
\fref{def-q*theta} is equimeasurable to $f$,  that is
$$f^* =  (f^{*\sigma})^*, \quad \mbox{or equivalently}\quad  \mu_f = \mu_{f^{*\sigma}}.$$ 
Moreover, $f^{*\sigma}$ is the only nonnegative function of $L^1(\Omega)$ which is a non increasing function of $\sigma(x)$ and is equimeasurable with $f$.
\end{Prop}

Now, we state our main inequalities, which are refinements of well-known Hardy-Littlewood type inequalities.  The proofs of the following  statements will be given in section \ref{proofth1}.

\begin{Thm}[A refined Hardy-Littlewood type inequality]
\label{thm1}
Let $\sigma$ be a real-valued  measurable function on  a domain $\Omega\subset \RR^d$, $d\geq 1$,  satisfying the assumptions of definition \ref{def1}, and consider the associated Jacobian function $a_\sigma$ defined  by 
\be
\label{athetaHL}
 \forall e \in \RR, \qquad a_{\sigma}(e) = \mbox{meas}\{ x\in \Omega, \  \sigma(x) < e\} \in [0,\mbox{meas}(\Omega)].
\ee 
and its pseudo-inverse $b_\sigma$  defined  on $]0,+\infty[$  by
\be
\label{btheta}
 \quad b_{\sigma}(\mu) = \sup\{ t<e_{max}:   a_\sigma(t) \leq \mu\}, \quad \forall \ \mu \geq 0,
\ee 
 with the convention that $b_{\sigma}(\mu) = -\infty$ if $\{ t<e_{max}:   a_\sigma(t) \leq \mu\} =\varnothing $.

Assume that the following function  $B_\sigma$ is well defined on $[0,\mbox{meas}(\Omega)[$, i. e.
\be
\label{Btheta}
B_\sigma (\mu) = \int_{a_\sigma (\sigma(x) )<\mu}   \sigma(x) dx \ \ \mbox{is finite}, \quad  \mbox{for all} \   \mu \in [0,\mbox{meas}(\Omega)[.
\ee
Then\\
i) The function $B_\sigma$ is convex on $[0,\mbox{meas}(\Omega)[$.\\
ii) Let
$$H_\sigma(\mu) = \inf_{0<s\leq\mu}  \frac{B_\sigma(\mu +s)+ B_\sigma(\mu-s)- 2 B_\sigma(\mu)}{s^2}, \quad \forall \mu>0,$$
and let $q$ be any nonnegative and nonzero  function $q\in L^1(\RR^d)$ such that
\be
\label{Kq*}
K(q^*, \sigma) =4 \int_0^{\|q\|_{L^{\infty}}} \frac{dt}{H_\sigma(\mu_q(t))} <\infty,
\ee
where $\mu_q=\mu_{q^*}$ is defined by \fref{mu}. 
 Then for  any nonnegative function $f\in L^1(\RR^d)$ we have the inequality
\be
 \begin{array}{ll}
 \label{ineq1}
\ds  \left( \|f-q^{*\sigma} \|_{L^1} +\|q\|_{L^1}-\|f\|_{L^1} \right)^2 & \leq \ds  K(q^*, \sigma)\left[ \int_{\Omega}  \sigma(x) (f(x)-q^{*\sigma}(x))dx\  \right.\\ &\ds \left.  \hspace{-2cm}+  \int_{0}^{+\infty}[\beta_{f^*,q^*}(s)b_\sigma(2\mu_q(s))- \beta_{q^*,f^*}(s)) 
b_\sigma(\mu_q(s))]ds \right],
  \end{array}
\ee 
 where $b_\sigma$ is the pseudo-inverse of $a_\sigma$ defined by \fref{btheta} and
 \be
 \label{betafg}
 \beta_{f,g}(s)= \mbox{meas}\{ x\in \Omega : f(x)\leq s <g(x)\},  \  \forall s\geq 0. 
 \ee
for all nonegative functions  $f,g \in L^1(\Omega).$

In particular the following refined Hardy-Littlewood type  inequality  holds for any nonnegative function $f\in L^1(\RR^d)$ such that the  quantity $K(f^*, \sigma)$ defined by \fref{Kq*} is finite 
\be
\label{ineq2}
\|f-f^{*\sigma}\|_{L^1}^2   \leq   K(f^*, \sigma) \int_{\Omega} \sigma(x) (f(x)- f^{*\sigma}(x)) dx.
\ee
\end{Thm}

Before making comments on the motivations of such a result, we just give some remarks and state a new functional inequality as a direct consequence of Theorem \ref{thm1} (see Corollary \ref{cor1} below).

\begin{Rk} 
\label{Rk3} In cases where the constant \fref{Kq*} is not finite, we still have the following estimate which will be shown in  the proof of theorem \ref{thm1} (see section \ref{proofth1}).
For  any nonnegative functions $f, q \in L^1(\RR^d)$ we have the inequality

\be
 \begin{array}{ll}
 \label{ineq11}
\ds  \int_{\Omega}  \sigma(x) (f(x)-q^{*\sigma}(x))dx\  \geq & \ds \int \left[B_\sigma(\mu_q(s) +\beta_{f,q}(s))+ B_\sigma(\mu_q(s) -\beta_{f,q}(s))- 2 B_\sigma(\mu_q(s))\right] ds  \\ &\ds \hspace{0cm}+ \int_{0}^{+\infty}[\beta_{q^*,f^*}(s)) b_\sigma(\mu_q(s))-\beta_{f^*,q^*}(s)b_\sigma(2\mu_q(s))]ds
  \end{array}
\ee 
where $\beta_{f,g}(s)$ is defined by \fref{betafg}.
Estimate \fref{ineq1} will be deduced from \fref{ineq11} when \fref{Kq*} is satisfied.
\end{Rk}

\begin{Rk} 
\label{Rk4}
When $e_{min}$ is finite, the function $B_{\sigma}(\mu)$ given by \fref{Btheta} is indeed finite and it  has another expression (see Lemma \ref{lem_atheta}, section \ref{prop-atheta}) which will be more practical for applications:
$$B_\sigma(\mu)= \int_0^\mu b_\sigma(s) ds,$$
where $b_\sigma$ is the pseudo-inverse of $a_\sigma$ defined by \fref{btheta}. 

\end{Rk}


Note that the constant $K$ in \fref{ineq1} only depends on $q^*$  and $\sigma$ and not on $f^*$. As we will see later in our stability analyses, the function $q$ will play the role of a steady state, and the dependence of $K$ on $q^*$ and $\sigma$ only,  is  important to allow general perturbations of the steady state $q$.
Before going to these stability issues, we give  a direct consequence of this theorem which will be proved in section \ref{proof-cor1}.

\begin{Cor}
\label{cor1}
Let  $d\in \NN$, $d\geq 1$. For any nonnegative function $u\in L^1(\RR^d)\cap L^{\infty}(\RR^d)$, and for all $0\leq m\leq d$, there holds
\be
\label{bathtub-classic}
\int_{\RR^d} |x|^m(u(x)- u^{*}(x)) dx \geq \frac{m}{4d} K_d^{-m/d}\|u\|_{L^1}^{-1+m/d}\|u\|_{L^\infty}^{-m/d}\|u-u^*\|_{L^1}^2,
\ee
where $K_d$ is the volume of the unit ball in $\RR^d$.
\end{Cor}

We now summarize heuristically  the general motivation of this work in very simple situations. It is well-known that rearrangement inequalities are very useful in many areas of mathematical analysis. In particular, rearrangements techniques are fundamental tools in the study of variational problems, 
and therefore in the stability analysis of a large class of models. One of these fundamental inequalities is the so-called Hardy-Littlewood inequality,  see \cite{HL1}, \cite{Lieb-Loss} \cite{Burch1}, which writes
\be
\label{HL-classique} \int_{\Omega} f(x) g(x) dx \leq  \int_{\Omega} f^*(x) g^*(x) dx,
\ee
where  $\Omega$ is a measurable subset of   $\RR^d$ and   where the rearrangements $f^*$ and $ g^*$ are defined on $\Omega$ by formula \fref{def-q*}. This inequality  implies the well-known particular form
\be
\label{HL-simple}
\int_{\Omega} |x| (f(x)- f^*(x)) dx \geq 0, 
\ee
which is commonly used in variational approaches to prove symmetry properties on the extremizers. A first consequence of our analysis  is that this inequality can be to extended to the generalized notion of rearrangements $f^{*\sigma}$ as follows:
\be
\label{HL-theta}
\int_{\Omega} \sigma(x) (f(x)- f^{*\sigma}(x)) dx \geq 0.
\ee
In fact, this extended inequality is in general not sufficient to provide quantitative stability results. Our main goal is to give a refinement of inequalities like \fref{HL-theta} in this general framework, to be used for {\em quantitative} stability issues. In others terms, we shall prove that quantities like the lhs of  \fref{HL-theta} 
controls the strong $L^1$ norm of $f-f^{*\sigma}$. More generally, we will derive from Theorem \ref{thm1} inequalities of the form
\be
\label{HL-theta-ref}
\int_{\Omega} \sigma(x) (f(x)-q^{*\sigma}(x)) dx + C_2(q^*, \sigma)\|f^*-q^*\|_{L^1} \geq C_1(q^*, \sigma)\left( \|f-q^{*\sigma} \|_{L^1} +\|q\|_{L^1}-\|f\|_{L^1} \right)^2,
\ee
which, in the particular case $f^*=q^*$, yields  a refinement of  the Hardy-Littelewood type inequality \fref{HL-theta}.  Now let us show heuristically why  an inequality like  \fref{HL-theta-ref}  directly provides  {\em quantitative  stability results} for appropriate evolution systems.
Assume  that $f=f(t,x)$ is a nonnegative solution to an evolution equation which preserves  the rearrangements in the variable $x$ and some energy functional (nonincreasing energy is sufficient in general), to wit
$$f(t)^*= f(0)^*, \qquad {\mathcal E}(f(t))=\int_{\Omega} \sigma(x) f(t,x) dx\leq {\mathcal E}(f(0)). $$ 
Of course the energy functionals we will have at hands are in general more complicated than this one, but we prefer to use this simplified energy to make clear this heuristical presentation. Consider a steady state $q$  which is
assumed to be a {\em fixed point} of the $\sigma$-rearrangement: $q^{*\sigma}=q$. Since the evolution problem preserves the rearrangement $f(t)^{*\sigma}$ and does not increase the energy,  we  have from \fref{HL-theta-ref}
\be
\label{ccxx}
\begin{array}{ll} \ds \int_{\Omega} \sigma(x) (f(0,x)(x)- q(x)) dx  &\geq  \ds \int_{\Omega} \sigma(x) (f(t,x)- q^{*\sigma}(x)) dx \\
&\hspace{-1cm}\ds \geq  C_1\left( \|f(t)-q^{*\sigma} \|_{L^1} +\|q\|_{L^1}-\|f(t)\|_{L^1} \right)^2- C_2 \|f(t)^*-q^*\|_{L^1}\\
&\hspace{-1cm} \ds = C_1\left( \|f(t)-q \|_{L^1} +\|q\|_{L^1}-\|f(0)\|_{L^1} \right)^2- C_2\|f(0)^*-q^*\|_{L^1}.
\end{array}\ee
The strong $L^1$ stability  is then deduced quantitatively from this inequality and  from the contractivity property of the rearrangements $\|f(0)^* - q^{*}\|_{L^1}\leq \|f(0) - q\|_{L^1}.$ 

Note that the stability in the energy space can also be derived.  In other terms, we can show from our inequality
\fref{HL-theta-ref} 
that the solution $f(t)$ remains also close  to $q$ in the energy norm, which means that the following quantity (in general $\sigma$ can be taken to be nonnegative)
$$\|f(t)-q \|_{L^1}+ \int_{\Omega} \sigma(x) |f(t,x)- q(x)| dx$$
remains small uniformly in time if it is small at time $0$.
Indeed, let us proceed by contradiction. Assume we have some $\e>0$, for which there exists a sequence $f_n(t)$ of solutions corresponding to initial data $f_n(0)$ satisfying
$$ \|f_n(0)-q \|_{L^1}+ \int_{\Omega} \sigma(x) |f_n(0)- q(x)| dx \to 0, \quad \mbox{as} \ n\to \infty,$$
and that there exist $t_n\geq 0$ such that $g_n=f_n(t_n)$ satisfies 
\be 
\label{ccxx1}\|g_n-q \|_{L^1}+ \int_{\Omega} \sigma(x) |g_n(x)- q(x)| dx \geq \e, \quad \forall n.
\ee 
From the $L^1$ stability shown above (see \fref{ccxx}), we have
$$ \sup_{t\geq 0} \|f_n(t)-q \|_{L^1} \to 0, \quad \mbox{as} \ n\to \infty,$$
which implies that $\|g_n-q \|_{L^1} \to 0$ and, up to a subsequence, $g_n$ converges to $q$ {\em a.e.} Therefore, applying the Br\'ezis-Lieb lemma  
\cite{Lieb-Loss} to the sequence $\sigma g_n$ we have
$$\|\sigma g_n- \sigma q \|_{L^1} + \|\sigma g_n \|_{L^1} - \|\sigma q \|_{L^1} \to 0.$$
Applying \fref{ccxx} to $f_n(t_n)=g_n$ we get $\|\sigma g_n \|_{L^1}\to  \|\sigma q \|_{L^1}$ from the assumption on  the initial perturbation in the energy space. Therefore
$$\|\sigma g_n- \sigma q \|_{L^1}  \to 0\quad  \mbox{and} \quad \|g_n- q \|_{L^1} \to 0.$$
This contradicts  \fref{ccxx1}
and the stability of $q$ in the energy space holds true.


\bigskip

We recall  that the constants  $ K$  in  \fref{ineq1} (or the constant $C(q^*, \sigma)$  in \fref{HL-theta-ref}) only depends on the rearrangement of the steady state $q$ and, as we shall see, this allows to avoid the usual assumptions made on the initial perturbations ($L^\infty$ bounded  or compactly supported  perturbations are usually considered), see for example the quantitative studies of the stability \cite{MP1,MP2} for 2D Euler and Valsov-Poisson systems. These two systems are indeed  well adapted to the present context  since they are known to preserve the rearrangements and some energy functionals (Hamiltonian, angular-momentum,etc).

We emphasize that, in the  the  particular case of the classical Schwarz rearrangements, our result can be interpreted as  a quantitative refinement of the well-known inequality \fref{HL-simple}, and in fact we have (see Corollary \ref{cor1} for a precise statement)
 \be
 \label{HL-simple-ref1}\int_{\Omega} |x| (f(x)- f^*(x)) dx \geq  C(\|f\|_{L^1}, \|f\|_{L^\infty}) \|f-f^*\|_{L^1}^2 .
 \ee
 A similar inequality  as \fref{HL-simple-ref1}, though in a particular context, has been proved by Pulvirenti and Marchioro \cite{MP1,MP2},  from which they directly derived the stability of a particular class of steady state solutions: particular steady state solutions  to  2D-Euler (planar couette flow and radial steady states), and spatially homogeneous steady states solutions to the Vlasov-Poisson system in a periodic domain in space. Our inequalities may be viewed as an extension of the inequalities in \cite{MP1,MP2} for generalized rearrangements and general energies $\sigma$.
 
\subsection{Application 1: a quantitative stability inequality for the  gravitational Vlasov-Poisson system}

Our approach will first  be applied to quantify  the non linear stability of a general class of steady state solutions to the gravitational Vlasov-Poisson (VP):
\be
\label{VP}
\begin{array}{l}
\ds \partial_t f + v\cdot \nabla_x f -  \nabla_x\phi_f \cdot \nabla_v f  =0,\\
\ds f(t=0,x,v)= f_{in}(x,v)\geq 0,
\end{array}
\ee
where $f=f(t,x,v)$ is the distribution function depending on time $t\geq 0$, on the position $x\in \RR^3$ and on the velocity $v\in \RR^3$. The quantities $\rho_f$ 
and $\phi_f$ are respectively the density and the Poisson field associated with $f$ according to
\be
\label{rho-phi}
\rho_f(x)= \int_{\RR^3} f(x,v) dv , \qquad  \phi_f(x) = -\frac{1}{4\pi |x|}*\rho_f.
\ee
The global Cauchy problem has been investigated in \cite{LP,Pf,S1} where unique global classical solutions ($C^1$ compactly supported) are derived. Global weak solutions have been also constructed in \cite{Ars,HorstHunze,Ill-Neun} and,  in particular, for all initial data $f_{in}$ in the energy space
\be
\label{E0}
{\mathcal E}_0 =\{ f\geq 0\quad \mbox{with} \quad f \in L^1\cap L^\infty (\RR^6)\quad  \mbox{and} \quad|v|^2 f \in L^1(\RR^6)\}
\ee
there exists a weak solution to \fref{VP}, which is also a renormalized solution, see \cite{dpl1,dpl2}.

 It is well known that these solutions preserve the equimeasurability in the following sense. If $f(t)$ is a solution at time $t$ to the VP system with initial data $f_{in}$, then one has
$$\mu_{f(t)}(\lambda) = \mu_{f_{in}}(\lambda), \qquad \forall \lambda \geq 0,$$
where $\mu_f(\lambda)$  is defined by
$$\mu_f(\lambda)=\mbox{meas}\{(x,v)\in \RR^6; f(x,v)>\lambda\}.$$
It is also well known that the classical solutions to VP preserve the  following energy functional (Hamiltonian)
\be
\label{Hamiltonien-sigma}
{\mathcal  H}(f)=\int \frac{|v|^2}{2} f dxdv  -  \frac{1}{2}  \int |\nabla_x \phi_{f}|^2dx.
\ee
which means
$${\mathcal  H}(f(t))={\mathcal  H}(f_{in}), \qquad \forall t\geq 0.$$
This last conservation property is replaced by an inequality in the case of weak solutions:  ${\mathcal  H}(f(t))\leq {\mathcal  H}(f_{in})$.

Finally consider the  classical family of compactly supported steady state solutions to (VP) of the form 
\be
\label{steady-state-VP} f_0(x,v)= F(e(x,v)),  \qquad e(x,v)= \frac{|v|^2}{2} + \phi_{f_0}(x),
\ee
where $F$ is a continuous function from $\RR$ to $\RR^+$ and  ${\mathcal C}^1$ strictly decreasing on the support of $Q$. More precisely, we assume that there exists $e_0<0$ such that
$F(e)=0$ for $e\geq e_0$ and $F$ is a ${\mathcal C}^1$ function on $(-\infty, e_0)$ with $F'<0$ on $(-\infty, e_0)$.
Let ${\mathcal E}$ be the following energy space:
\be
\label{E}
{\mathcal E} = \{ f=f(x,v) \geq 0,  \quad \|f\|_{\mathcal E}= \|(1+|v|^2) f\|_{L^1} + \|\nabla_x\phi_f\|_{L^2}^2 <\infty\},
\ee
which contains the above space ${\mathcal E}_0$ given by \fref{E0}, since one has the classical interpolation inequality
\be  
\label{interpo-VP}\|\nabla_x\phi_f\|_{L^2} \leq C \||v|^2 f\|_{L^1} ^{1/2} \|f\|_{L^1} ^{7/6}\|f\|_{L^\infty}^{1/3},
\ee
where $C$ is a universal constant.
Finally, for $f\in {\mathcal E}$, we denote by $f^*$ the standard Schwarz rearrangement of $f$ in the phase
space $(x, v) \in \RR^6$. 
%
%

With these notations, we now state  our {\em quantitative stability inequalities} for the gravitational Vlasov-Poisson system. The proof of the following theorem is given in section  \ref{sectionVPP}.

\begin{Thm}[Control of $f$ by the Hamiltonian and its Schwarz rearrangement]
\label{thm-VPG}
 Let $f_0$ be given by \fref{steady-state-VP} with $F$ satisfying the above assumptions.  Then we have the following.
 \begin{itemize}
 \item[i)] {\em Global control}. There exist a constant $K_0 > 0$ depending only
on $f_0$ such that or all $f\in {\mathcal E}$
\be
\label{ineq-grav-glob}
\begin{array}{ll}
 \ds \|f-f_0\|_{L^1}  \leq & \|f^* - f_0^*\|_{L^1}+ \\ & \ds K_0\left[ {\mathcal H}(f) - {\mathcal H}(f_0)  + 2|\phi_{f_0}(0)| \|f^* - f_0^*\|_{L^1} + \|\nabla \phi_{f}-\nabla {\phi_{f_0}}\|_{L^2}^2\right]^{1/2}.
 \end{array}
\ee
\item[ii)] {\em Local control}. There exist constants $K_0,  R_0 > 0$  depending only
on $f_0$ and a continuous map $\phi \mapsto z_\phi$ from ($\dot H^1, \|\cdot\|_{\dot H^1})$ to $\RR^3$ such that, for all $f\in {\mathcal E}_0$ satisfying
\be
\label{local-assumption}\inf_{z\in \RR^3}\left ( \|\phi_{f}-{\phi_{f_0}}(\cdot-z)\|_{L^\infty} + \|\nabla \phi_{f}-\nabla {\phi_{f_0}}(\cdot-z)\|_{L^2}\right)<R_0,
\ee
  there holds:
\be
\label{ineq-grav-loc}
\begin{array}{l}
\ds \|f-f_0(\cdot-z_{\phi_f})\|_{L^1}+ \|\nabla \phi_{f}-\nabla {\phi_{f_0}}(\cdot-z_{\phi_f})\|_{L^2}^2 \leq \ds  \|f^* - f_0^*\|_{L^1}+ \\ \hspace{5.5cm} \ds K_0
\left[ {\mathcal H}(f) - {\mathcal H}(f_0) + K_0\|f^* - f_0^*\|_{L^1} \right]^{1/2}.
\end{array}
\ee
where we denoted $\phi_{0}(\cdot-z_{\phi_f})(x)= \phi_{f_0}(x-z_{\phi_f})$ and $f_0(\cdot-z_{\phi_f})(x,v)=
f_0(x-z_{\phi_f},v).$
\end{itemize}
\end{Thm}
\begin{proof}[{ Proof of the nonlinear stability of $f_0$ using Theorem \ref{thm-VPG}}] Inequality \fref{ineq-grav-loc} directly implies the nonlinear stability (in the energy norm) of any steady state of the Vlasov-Poisson system having the form \fref{steady-state-VP} with  the above suitable assumptions on $F$. The stability proof will not use any  compactness argument and this provides a completely quantitative statement of  the stability result of \cite{LMR}.   To get this quantitative result, inequality  \fref{ineq-grav-loc} will be just combined with the fact that the Vlasov-Poisson system preserves the rearrangement and does not increase the Hamiltonian  of the distribution function. Indeed, let $M>0$ be fixed (but arbitrary) and choose an initial data $f(0)\in {\mathcal E}_0$ such that $\|f(0)\|_{L^\infty}+\|f(0)\|_{\mathcal E}\leq M + \|f_0\|_{L^1} $.  A solution $f(t)$ to \fref{VP} (in a weak or classical sense) satisfies  ${\mathcal H}(f(t)) \leq {\mathcal H}(f(0))$.  Combining this with inequality \fref{interpo-VP} ensures that $\|f(t)\|_{{\mathcal E}}\leq C(M)$, for all $t\geq 0$, where the constant $C(M)$ depends only on $M$. Now, by interpolation inequalities we have
\be
\label{control11}\| \phi_{f(t)}-\phi_{f_0}\|_{L^\infty} + \|\nabla \phi_{f(t)}-\nabla {\phi_{f_0}}\|_{L^2}\leq C  \|f(t)-f_0\|_{L^1},
\ee
(see \cite{LMR}, page 189, for a proof of this inequality), where again the constant $C$ depends only on $M$.
This inequality still holds for any translated $f_0$ in space, in particular
\be
\label{control1}
\inf_{z\in \RR^3} \left(\| \phi_{f(t)}-\phi_{f_0}(\cdot -z)\|_{L^\infty} + \|\nabla \phi_{f(t)}-\nabla {\phi_{f_0}}
(\cdot -z)\|_{L^2}\right) \leq C  \inf_{z\in \RR^3} \left(\|f(t)-f_0(\cdot -z)\|_{L^1}\right),
\ee
 We then take initial data $f(0)$ such that 
\be
\label{ineqq1}\|f(0) - f_0\|_{L^1}+  K_0
\left[ {\mathcal H}(f(0)) - {\mathcal H}(f_0) + C_0\|f(0) - f_0\|_{L^1} \right]^{1/2}< R_0/2C.
\ee
From the contractivity property of the rearrangements, we also have
$$\|f(0)^* - f_0^*\|_{L^1}+  K_0
\left[ {\mathcal H}(f(0)) - {\mathcal H}(f_0) + C_0\|f(0)^* - f_0^*\|_{L^1} \right]^{1/2}< R_0/2C,$$
which implies from the conservation of the rearrangement and the decrease of the Hamiltonian, that  for all $t\geq 0$
\be
\label{ineqq2}
\|f(t)^* - f_0^*\|_{L^1}+  K_0
\left[ {\mathcal H}(f(t)) - {\mathcal H}(f_0) + C_0\|f(t)^* - f_0^*\|_{L^1} \right]^{1/2}< R_0/2C.
\ee
We now claim that, for such initial data, we have
\be
\label{claim-control}
\inf_{z\in \RR^3} \left(\|f(t) - f_0(\cdot -z)\|_{L^1}\right)< R_0/2C, \quad \forall t\geq 0.
\ee
Indeed, assume the contrary. Then, from \fref{ineqq1} and the continuity in time of the Vlasov-Poisson flow, there exists $t_0\geq 0$ such that
$\inf_{z\in \RR^3} \left(|f(t_0) - f_0(\cdot -z)\|_{L^1}\right)= 2R_0/3C < R_0/C.$ This implies from \fref{control1} that 
\fref{local-assumption} holds for $f(t_0)$
$$\inf_{z\in \RR^3} \left(\| \phi_{f(t_0)}-\phi_{f_0}(\cdot -z)\|_{L^\infty} + \|\nabla \phi_{f(t_0)}-\nabla {\phi_{f_0}}
(\cdot -z)\|_{L^2}\right) < R_0.$$
Therefore, the control \fref{ineq-grav-loc} holds for $f(t_0)$. Combining this with \fref{ineqq2} yields
$$2R_0/3C=\|f(t_0) - f_0\|_{L^1}<R_0/2C,$$
which is a contradiction. We conclude that \fref{claim-control} holds true and then from \fref{control1}
$$\inf_{z\in \RR^3} \left(\| \phi_{f(t)}-\phi_{f_0}(\cdot -z)\|_{L^\infty} + \|\nabla \phi_{f(t)}-\nabla {\phi_{f_0}}
(\cdot -z)\|_{L^2}\right)< R_0,   \quad \forall t\geq 0.$$
This ensures that inequality \fref{ineq-grav-loc} is satisfied by the solution $f(t)$ for all $t\geq 0$, and the nonlinear stability is an immediate consequence of this control.
\end{proof}

\subsection{Application 2: quantitative stability inequalities for the  2D Euler system}
We consider the incompressible  2D Euler system on a domain $\Omega$ of $\RR^2$  written in  the so-called vorticity formulation
\be\label{euler-vorticity}
\left\{\begin{array}{l}\pa_t\omega+\nabla^{\perp}\psi\cdot\nabla \omega=0, \quad \mbox{on} \ \RR_+\times \Omega,\\
\ds - \Delta \psi(x) = \omega(x), \quad \mbox{for} \  x\in \Omega,\\
\omega(0,x)=\omega_{in}(x),  \quad \mbox{on} \  \Omega.
\end{array}
\right.
\ee 
This formulation needs  suitable boundary conditions to be well-posed. We shall consider this equation in the following cases:
\begin{itemize}  
\item Radially symmteric domain
$\Omega =D(0,R)$   where $D(0,R)$ is the open disc of $\RR^2$ centered at $0$ with  radius $R>0$. 
In this case equation \fref{euler-vorticity} is usually complemented by  a Dirichlet boundary condition on $\Omega$:
\be 
\label{bc-Omega} 
- \Delta \psi = \omega \quad \mbox{on} \quad  \Omega, \qquad \psi(x)= 0 \quad \mbox{on} \quad \pa \Omega.
\ee
If $R=+\infty$ then $\psi$ is given by the formula
\be
\label{psi-R2}
\psi(x) = \frac{1}{2\pi} \int_{\RR^2} \ln(|x-y|)\omega (y) dy, \quad \mbox{for} \  x\in \RR^2.
\ee
We assume that $\omega_{in}\in L^1\cap L^\infty$ which ensures from the Yudovich theorem  \cite{Yud} the existence and uniqueness  of weak solutions $u\in L^\infty(\RR_+, L^1\cap L^\infty)$ to \fref{euler-vorticity}-\fref{bc-Omega}. We also assume that the initial vorticity $\omega_{in}$ is nonnegative, which implies that the vorticity
$\omega(t,\cdot)$ is nonnegative for all time.  Note that elliptic regularity implies  that $\psi$ is a $C^1$ function of $x\in\overline \Omega$.
It is well known that any solution to 
\be
\label{steady-Fphi}
\omega_0(x) = F(\psi_0(x)); \qquad -\Delta \psi_0 = \omega_0,\ee
where $F$ is a given smooth  and nonnegative function on $\RR$, is a steady state solution to \fref{euler-vorticity}.
 Since $\Omega$ has a radial symmetry, the Gidas-Ni-Nirenberg theorem ensures that all positive steady states  $\omega_0$
given by \fref{steady-Fphi}  have radial symmetry:
\be
\label{steady-radial}\omega_0(x)= G(|x|), \qquad \forall x \in \Omega.
\ee
To get  a stability inequality for these steady states, we will use the following momentum functional
\be
\label{energy-radial-euler} 
{\mathcal A}(\omega)= \int_{\Omega} |x|^2 \omega(x) dx,
\ee 
which is preserved by the flow  \fref{euler-vorticity} within the spherical symmetry context.

\item Rectangular domain  $\Omega = ]0,L_1[\times ]0,L_2[$ where $L_1, L_2>0,$  and  where  $L_2$ may be infinite. Consider the vorticity equation
\fref{euler-vorticity} with a periodic boundary condition in the $x_1$ direction and a Dirichlet boundary condition in the $x_2$ direction:
$$\begin{array}{l}\ds \Delta \psi (x) = - \omega(x), \quad \forall x\in]0,L_1[\times ]0,L_2[,\\
\ds  \psi(x_1,0)=0, \quad \psi(x_1,L_2)=0, \quad \forall x_1\in [0,L_1]\\
\ds  \quad  \psi(0,x_2)=  \psi(L_1,x_2),  \quad \forall x_2 \in  [0,L_2].
\end{array}$$
In case where $L_2=+\infty$ the boundary condition at $x_2=L_2$ has to be replaced by a vanishing condition at infinity: $\lim_{x_2\to +\infty} \psi(x_1,x_2)=0,  \quad \forall x_1\in [0,L_1].$

Let $\omega_0$ be the following steady state
\be
\label{ss-x2}
\omega_0(x) = F(x_2),  \quad x\in \Omega,
\ee
where $F$ is a monotonic  function such that $\omega_0 \in L^1\cap L^\infty$. We only consider the case where  $F$ is decreasing, since the  extension to increasing functions $F$ (in the case where $\Omega$ is bounded, $L_2<+\infty$) is straightforward. 

To get  a stability inequality for these steady states, we will use the following momentum functional
\be
\label{energy-rect-euler} 
{\mathcal B}(\omega)= \int_{\Omega} x_2 \omega(x) dx,
\ee 
which is preserved by the flow  \fref{euler-vorticity} within the rectangular symmetry case.

\item General bounded domains: let $\Omega$ be a bounded domain of class ${\mathcal C}^{2,\alpha}$ for some $\alpha \in (0,1)$.  In this case, we shall take a Dirichlet boundary condition for the stream function $\psi$, that is: $\psi (x)=0$ on the boundary  $\pa\Omega$, and consider the following family of  steady state solutions
\be
\label{SS-phi-euler} \omega_0 = F(\psi_0)\quad \mbox{with} \quad - \Delta \psi_0=\omega_0 , \quad \mbox{and}\quad  \psi_0=0\quad  \mbox{on} \quad \pa\Omega,
\ee
where $F$ is nonnegative, continuous and non increasing function on $\RR$ with $F(0)>0$. 
To get  a stability inequality for these steady states, we will use the following kinetic energy  functional
\be
\label{energy-Ham-euler} 
{\mathcal H}(\omega)= \frac{1}{2}\int_{\Omega}  \psi(x) \omega(x) dx = \frac{1}{2}\int_{\Omega}  |\nabla \psi(x)|^2 dx \ee 
which is preserved by the flow  \fref{euler-vorticity}.

\end{itemize}

The $L^1$ stability of radial steady  states \fref{steady-radial}  has been proved  in \cite{MP1} in the case of a bounded 
domain and $L^\infty$ perturbations. As pointed by the authors in  \cite{MP1}, the case where the domain is not bounded cannot be covered by their technique  unless one restricts to compactly supported perturbations. The $L^1$ stability of steady states of the form  \fref{ss-x2}  has also been proved  in  \cite{MP1}  with the same restrictions on the domain and on the perturbations.
In the following statement, we show that the functional inequality \fref{ineq1} stated in Theorem \ref{thm1} provides a quantitative estimate of the perturbation for all time in terms of the initial perturbation. Note that the domain $\Omega$ may be unbounded and the perturbations are not supposed to be in $L^\infty$.

\begin{Thm}
\label{thm-Euler1}[Stability inequalities for 2D Euler with symmetries]
Assume that   $\omega(t) \in L^1(\Omega)$ is a global weak solution at time $t$ to the 2D Euler equation in the vorticity form \fref{euler-vorticity}, with a nonnegative  initial data $\omega_{in}\in L^1(\Omega)$. 
\begin{itemize}
\item Radial steady states. Assume that $\Omega$ is a radially symmetric domain $B(0,R)$ where $R$ may be infinite. Let $q\in L^1\cap L^\infty(\Omega)$ be a steady state of the form \fref{steady-radial} where $G$ is a nonnegative decreasing function such that $\int_\Omega |x|^2G(|x|)dx<+\infty$.  
Then, for all $ t\geq 0$, we have the following quantitative stability inequality:   
\be
\label{ineq-E2D-radial}
\begin{array}{ll}
\ds \left( \|\omega(t)-q\|_{L^1(\Omega)}+\|q\|_{L^1}-\|\omega_{in}\|_{L^1} \right)^2 \leq &\ds  4\pi{\|q\|_{L^\infty}} \left[ \int_{\Omega} |x|^2(\omega_{in}(x)-q (x)) dx\right. \\ 
\ds &\ds \hspace{1.5cm} \left.  +\frac{2}{\pi} \int_0^{+\infty} \mu_q(s) \beta_{\omega_{in}^*,q^*}(s) ds\right].
\end{array}
\ee
where $\mu_q$  and $ \beta_{\omega_{in}^*,q^*}$ are defined by \fref{mu} and \fref{betafg}.
 In particular, if $q$ is compactly supported then
\be
\label{ineq-E2D-radial-c}
\begin{array}{ll}
\ds \left( \|\omega(t)-q\|_{L^1(\Omega)}+\|q\|_{L^1}-\|\omega_{in}\|_{L^1} \right)^2 \leq &\ds  4\pi{\|q\|_{L^\infty}} \left[ \int_{\Omega} |x|^2(\omega_{in}(x)-q (x)) dx\right. \\ 
\ds &\ds \hspace{1.5cm} \left.  +\frac{2}{\pi} \mbox{meas}(\mbox{Supp}(q))\|\omega_{in}^*-q^*\|_{L^1(\Omega)}\right].
\end{array}
\ee
\item Steady states on a rectangular domain. Assume that $\Omega$ is a rectangular domain $\Omega= ]0,L_1[\times ]0,L_2[$ where $L_1, L_2>0,$  and  where $L_2$ may be infinite.  Let $q\in L^1\cap L^\infty(\Omega)$ a steady state solution to the 2D-Euler system \fref{euler-vorticity} of the form \fref{ss-x2} where $F$ is a nonnegative decreasing function such that $\int_0^{L_2} x_2F(x_2) dx_2<+\infty$.  Then for all $ t\geq 0$, we have
\be
\label{ineq-E2D-rectangular}
\begin{array}{ll}
\ds\left( \|\omega(t)-q\|_{L^1}  + \|q\|_{L^1}-\|\omega_{in}\|_{L^1} \right)^2\leq  &\ds 4{\|q\|_{L^\infty}}\left[ \int_{\Omega} x_2 (\omega_{in}(x)- q (x)) dx\right. \\ 
&\ds \hspace{1.3cm} \left.  +\frac{2}{L_1} \int_0^{+\infty} \mu_q(s) \beta_{\omega_{in}^*,q^*}(s) ds \right].
\end{array}
\ee
In particular, if $q$ is compactly supported then
\be
\label{ineq-E2D-rectangular-c}
\begin{array}{ll}
\ds\left( \|\omega(t)-q\|_{L^1}  + \|q\|_{L^1}-\|\omega_{in}\|_{L^1} \right)^2\leq  &\ds 4{\|q\|_{L^\infty}}\left[ \int_{\Omega} x_2 (\omega_{in}(x)- q (x)) dx\right. \\ 
&\ds \hspace{1.3cm} \left.  +\frac{2}{L_1} \mbox{meas}(\mbox{supp}(q))\|\omega_{in}^*-q^*\|_{L^1}\right].
\end{array}
\ee
\end{itemize}

\end{Thm}

The proof of this theorem is given in section \ref{sectionEuler1P}.  We now briefly show how Theorem \ref{thm-Euler1} directly implies  the stability of steady states  of the form \fref{steady-radial} or \fref{ss-x2}.
We just present the argument  for the steady steady states \fref{steady-radial}, the proof for  \fref{ss-x2} is similar.

\medskip
\noindent{\em Proof of a stability result using Theorem \ref{thm-Euler1}.}   Let $q$ be a steady state solution to the 2D Euler equation \fref{euler-vorticity} of the form \fref{steady-radial} such that $\int_{\Omega} |x|q(x)dx < +\infty.$ We want to prove
that if a sequence of initial data $\omega_{in}^n$ satifies
$$   \|\omega_{in}^n - q\|_{L^1} \to 0 \ \ \mbox{and}\  \ \int_{\Omega} |x|^2 |\omega_{in}^n(x) - q(x)| dx \to 0, \quad \mbox{as} \ n\to +\infty,$$
then any corresponding weak solution $\omega^n(t,x)$ to \fref{euler-vorticity} satisfies:
$$   \sup_{t\geq 0} \|\omega^n(t) - q\|_{L^1} \to 0 \ \ \mbox{and}\  \   \sup_{t\geq 0} \int_{\Omega} |x|^2 |\omega^n(t,x) - q(x)| dx \to 0.$$ 
The case of compactly supported $q$  is a straightforward consequence of \fref{ineq-E2D-radial-c} and the contractivirt property of the rearrangement  $\|f^*-g^*\|_{L^1}\leq \|f-g\|_{L^1}$.  We then assume  that $R=+\infty$ and that the support of $q$ may be unbounded.
We apply inequality \fref{ineq-E2D-radial} to $\omega^n(t)$ and clearly see that it suffices to prove
that 
\be
\label{toprove} \int_0^{+\infty}  \mu_q(s) \beta_{(\omega_{in}^{n})^*, q^*}(s) ds \to 0, \quad \mbox{as} \quad n\to +\infty.
\ee
To prove \fref{toprove} we first observe that 
$$ \int_0^{+\infty} \beta_{(\omega_{in}^{n})^*, q^*}(s) ds = \int_{\Omega} (q^*(x)-(\omega_{in}^n)^*(x))_+ dx\leq 
\|(\omega_{in}^n)^* - q^*\|_{L^1}\leq \|\omega_{in}^n - q\|_{L^1} \to 0.$$
This means that $\mu_q(s) \beta_{(\omega_{in}^{n})^*, q^*}(s)\to 0$ for almost every $s>0$. It is the sufficient
to  suitably dominate this quantity in order to get \fref{toprove}. In fact we clearly have  $\mu_q(s) \beta_{(\omega_{in}^{n})^*, q^*}(s)\leq \mu_q(s)^2$ and we claim that
\be
\label{clalam}
\int_0^{+\infty} \mu_q(s)^2 ds <+\infty.
\ee
This claim is proved as follows
\begin{align*}\int_{\RR^2} |x|^2q(x) dx &=  \int_\RR |x|^2 \left(\int_0^{q(x)} ds\right) dx =
  \int_0^{+\infty} \left(\int_{q(x)>s} |x|^2dx\right) ds \\
&= \int_0^{+\infty} \left(\int_{q^\sharp(\pi |x|^2)>s} |x|^2dx\right) ds = \int_0^{+\infty} \left(\int_{\mu_q(s)>\pi |x|^2} |x|^2dx\right) ds\\
&= \frac{1}{2\pi} \int_0^{+\infty} \mu_q(s)^2 ds.
\end{align*}
This proves \fref{clalam} which implies also \fref{toprove}.  The proof of the stability statement is then complete.

\medskip

We  now give a  stability inequality for another class of steady states solutions to the 2D Euler equation  in case of a bounded domain with no  specific symmetries.  This class is given by
\fref{SS-phi-euler} where $F$ is a decreasing function.

\begin{Thm}
\label{thm-Euler2}[Stability inequalities for 2D Euler on bounded domains]
 Let  $\Omega$ be a bounded  domain  of $\RR^2$ of class ${\mathcal C}^{2,\alpha}$ for some $\alpha\in (0,1)$. Let $F: \RR \to \RR$ be a nonnegative function which is continuous and decreasing on its support. Assume that  $F(0) >0$ and consider  $\omega_0$  satisfying 
  $$\omega_0 (x)= F(\psi_0(x)), \quad \mbox{with} \qquad -\Delta \psi_0=\omega_0, \quad  \quad \mbox{and}\quad  \psi_0=0 \ \mbox{on}\  \pa \Omega.$$   Then\\
 \noindent i) the function
 $$\Psi_0(\mu)= \int_0^\mu \psi_0^\sharp(\mbox{meas}(\Omega)-s) ds, \quad \mbox{for all}\  \ 0\leq \mu\leq \mbox{meas}(\Omega),$$
 is strictly convex, and\\
 \noindent  ii)  there exists a positive constant $C$ such that the following holds.  For all nonnegative function $\omega \in L^1(\Omega)$, we have the following inequality
\be
\label{ineq-E2D-domain}
\begin{array}{l}
{\mathcal H}(\omega) - {\mathcal H} (\omega_0) +\|\psi_0\|_{L^\infty} \|\omega^* - \omega_{0}^*\|_{L^1}\ds \geq  \|\nabla \psi- \nabla \psi_0\|^2_{L^2} + \\ \ds \hspace{2cm} \int_0^{+\infty}  \left[\Psi_0 (\mu_0(s)+ \beta_{\omega}(s))+ \Psi_0 (\mu_0(s)- \beta_{\omega}(s)) -2\Psi_0 (\mu_0(s))\right] ds
\end{array}
\ee
where, for all $s>0$ 
$$\mu_0(s)= \mbox{meas} \{ x\in \Omega : \omega_0(x)> s\}, \qquad \beta_\omega (s)= \mbox{meas} \{ x\in \Omega :  \omega(x)\leq  s< \omega_0(x)\},$$
and where
$${\mathcal H}(\omega)= \frac{1}{2}\int_{\Omega}  \psi(x) \omega(x) dx= \frac{1}{2}\int_{\Omega} 
|\nabla \psi(x)|^2 dx$$
with 
$$ -\Delta \psi=\omega , \quad \mbox{and}\quad  \psi=0\quad  \mbox{on} \quad \pa\Omega.$$
\end{Thm}

\medskip

The proof of this theorem is given in section \ref{proof-Euler2}.

\medskip
\noindent{\em Proof of a stability result using Theorem \ref{thm-Euler2}.}
To prove the stability of a steady state $\omega_0$ for the 2D Euler flow given in Thoerem \ref{thm-Euler2}, it is sufficient to prove that:
if   $\omega_n$ is a sequence of functions  which converges to $\omega_0$ in the following sense
\be
\label{hyp-conv}
 \mathcal{H}(\omega_n) \to \mathcal{H}(\omega_0) \quad \mbox{and}\quad \|\omega_n^*- \omega_0^*\|_{L^1} \to 0,
 \ee
then, up to a subequence extraction, we have 
\be
\label{res-conv}
 \|\omega_n- \omega_0\|_{L^1} \to 0 \quad  \mbox{and} \quad
    \|\nabla \psi_n- \nabla\psi_0\|_{L^2}^2 \to 0 .\ee
Let $\omega_n$ be a sequence satisfying \fref{hyp-conv}, then from \fref{ineq-E2D-domain} and the convexity of $\Psi_0$ we have
$$ \int_0^{+\infty}  \left[\Psi_0 (\mu_n(s)+ \beta_{\omega_n}(s))+ \Psi_0 (\mu_n(s)- \beta_{\omega_n}(s)) -2\Psi_0 (\mu_n(s))\right] ds  \to 0, \quad \mbox{as} \ n\to +\infty,$$
where $\mu_n(s)= \mbox{meas} \{ x\in \Omega : \omega_n(x)> s\}.$
Therefore the convexity of $\Psi_0$ implies
$$\Psi_0 (\mu_n(s)+ \beta_{\omega_n}(s))+ \Psi_0 (\mu_n(s)- \beta_{\omega_n}(s)) -2\Psi_0 (\mu_n(s))   \to 0,  \quad \mbox{when} \ n\to +\infty, \quad \mbox{for} \ \  a. e. \ \ s\geq 0, $$
up to subsequence extraction. Now, since  $\beta_{\omega_n}$ is a bounded sequence ($\beta_{\omega_n}(s)\leq \mu_0(s)\leq \mbox{meas} (\Omega)),$ and since $\Psi_0$ is  strictly convex,  any accumulation point of the sequence $\beta_{\omega_n}(s)$ must be equal to $0$. Therefore
$$ \beta_{\omega_n}(s) \to 0, \quad  \mbox{for} \ \  a. e. \ \ s\geq 0.$$
We then integrate this identity (using dominated convergence) and get
\be
\label{betanto0} \int_0^\infty \beta_{\omega_n}(s) ds \to 0.\ee
Now, we introduce $\alpha_{\omega_n}(s)=\mbox{meas} \{ x\in \Omega :  \omega_0(x)\leq  s< \omega_n(x)\}$.
We clearly have
$$ \int_0^\infty (\alpha_{\omega_n}(s) + \beta_{\omega_n}(s))ds = \int_{\Omega} (\omega_n(x)- \omega_0(x))_+dx +\int_{\Omega} (\omega_0(x)- \omega_n(x))_+dx = \|\omega_n- \omega_0\|_{L^1},$$
and 
$$ \begin{array}{ll}
\ds \int_0^\infty (\alpha_{\omega_n}(s) - \beta_{\omega_n}(s))ds &\ds = \int_{\Omega} (\omega_n(x)- \omega_0(x))_+dx -\int_{\Omega} (\omega_0(x)- \omega_n(x))_+dx \\ &\ds =\|\omega_n\|_{L^1}- \|\omega_0\|_{L^1} 
= \|\omega_n^*\|_{L^1}- \|\omega_0^*\|_{L^1}.\\
\end{array}$$
Thus, from \fref{betanto0} 
$$\int_0^\infty \beta_{\omega_n}(s) ds = \frac{1}{2}\left(\|\omega_n- \omega_0\|_{L^1} + \|\omega_0^*\|_{L^1}- \|\omega_n^*\|_{L^1}\right)\to 0 \quad \mbox{as} \quad n\to +\infty.$$
From  $\|\omega_n^*- \omega_0^*\|_{L^1} \to 0$ and \fref{betanto0} we then get $\|\omega_n- \omega_0\|_{L^1} \to 0$.
This concludes the proof of the claimed stability property.

\section{Proofs of the functional inequalities}
The goal of this section is to prove Theorem \ref{thm1}. We proceed in several steps. We first study the properties  of the Jacobian $a_\sigma$ defined by \fref{atheta}, then give the proof  of Proposition \ref{f*theta-equi} , and finally  deal with the proof of Theorem \ref{thm1} and Corollary \ref{cor1}.
\subsection{Properties of the function $a_\sigma$.}
\label{prop-atheta}
We start by giving some useful properties of the function $a_\sigma$ defined by \fref{atheta}. We have the following lemma.
\begin{Lemma}
\label{lem_atheta}  Assume the assumptions made on $\sigma$ in Definition \ref{def1} to be satisfied.  Then
\begin{enumerate}[ label=\roman*)]
\item For all $e<e_{max} $, we have $a_\sigma(e) <  \mbox{meas} (\Omega)$. Moreover, the function $a_\sigma$ is nondecreasing and  continuous from $]-\infty, e_{max}[$ to $[0,\mbox{meas}(\Omega)[$ with $\lim_{e\to e_{min}} a_\sigma(e)= 0.$ 
In particular
 $a_{\sigma}(e)=0$ for all $e\leq e_{min}$ if  $e_{min}$ is finite, and $a_\sigma(e) = \mbox{meas} (\Omega)$ for all $e\geq e_{max}$ if $e_{max}$ is finite.
\item If  $b_\sigma$  is the pseudo-inverse of $a_\sigma$ defined  by \fref{btheta}, then
 \be\label{inv-droite} a_\sigma \circ b_{\sigma} (\mu) = \mu,  \quad \forall \mu \in [0,\mbox{meas}(\Omega)].\ee
 Moreover, for all  $e\in ]-\infty, e_{max}[$ and $ \mu\geq 0$,  we have
\be
\label{equiv-btheta} a_\sigma(e)\leq \mu \Longleftrightarrow e \leq  b_\sigma(\mu). 
\ee
\item The level sets of the function $a_\sigma \circ \sigma$ are of zero measure:
\be
\label{alkalam}
\mbox{meas}  \{x\in \Omega;\  a_\sigma(\sigma(x))= \mu \ \mbox{and} \  \sigma(x)<e_{max}\} = 0, \quad \forall \mu \geq 0. 
\ee
\item When $e_{min}$ is finite, the function $B_{\sigma}(\mu)$ given by \fref{Btheta} is well defined and we have
\be
\label{Btheta-simple} 
B_\sigma(\mu)= \int_0^\mu b_\sigma(s) ds.
\ee
where $b_\sigma$ is the pseudo-inverse of $a_\sigma$ defined by \fref{btheta}.
\end{enumerate}
\end{Lemma}

\begin{proof}
We first prove assertion {\em i)}. The fact that  $a_\sigma(e) <  \mbox{meas} (\Omega)$ whenever $e<e_{max} $ comes from the definition \fref{A2} of 
$e_{max} $ and the monotonicity of $a_\sigma$. Let us prove the continuity of $a_\sigma$. In fact assumption \fref{levelset0} on $\sigma$ is equivalent to the continuity of $a_\sigma$ on $ ]-\infty, e_{max}[$.
 Indeed, assume first the continuity of $a_\sigma$ on $]-\infty, e_{max}[$ and take $e\in ]-\infty, e_{max}[$. For all  sufficiently small $\e >0$  we have
$$a_\sigma (e+\e) \geq \mbox{meas}\{x\in \Omega;  \sigma(x)=e\}  + a_\sigma(e) .$$
Letting $\e \to 0$ and using the continuity of $a_\sigma$ we get $\mbox{meas}\{x\in \Omega;  \sigma(x)=e\}=0$ as claimed. Conversely, assume that all the level sets $\{x\in \Omega;  \sigma(x)=e\}$  are of zero measure for all  $e\in ]-\infty, e_{max}[$, then 
$a_\sigma (e+\e)- a_\sigma (e)= \mbox{meas}\{x\in \Omega; e\leq \sigma(x)< e+\e\} $  goes to 
$\mbox{meas}\{x\in \Omega;  \sigma(x)=e\}=0$ when $\e \to 0$. This gives the right continuity of $a_\sigma$. Since $a_\sigma$ is always left continuous, we then deduce its continuity as claimed. 
We now prove that $\lim_{e\to e_{min}} a_\sigma(e)= 0$.  If $e_{min}>-\infty$ then for sufficiently small $\e>0$, 
$a_\sigma( e_{min}+ \e)$ is finite, and tends to $a_\sigma( e_{min})$ when $\e\to 0$ (from the continuity of $a_\sigma$).
But from the definition of $e_{min}$ we have $a_\sigma( e_{min})=\mbox{meas}\{x\in \Omega;  \sigma(x)<e_{min}\}=0$ and this yields $\lim_{e\to e_{min}} a_\sigma(e)= 0$.  Assume now $e_{min}=-\infty$. We know that $a_\sigma(e)$ is finite for 
$e<e_{max}$ and that $\mathds{1}_{\sigma(x)<e} \leq \mathds{1}_{\sigma(x)<e_0}$ for all $e\leq e_0.$ Therefore by dominated convergence we get $\lim_{e\to -\infty} a_\sigma(e)= 0$. The rest of assertion {\em i)} is elementary the proof of which is left to the reader.

Now we prove assertion {\em ii)}. We start by proving \fref{inv-droite} and first consider $\mu \in ]0, \mbox{meas}(\Omega)[.$  Since $\lim_{e\to e_{max}^-} a_\sigma(e)=\mbox{meas} (\Omega) $, there exists $t_0<e_{max}$ such that $a_\sigma(t_0) > \mu$, which implies that $a_\sigma(t) > \mu$ for all $t\geq t_0$. Therefore, from the definition \fref{btheta}
of $b_\sigma(\mu)$ we have $b_\sigma(\mu) \leq t_0<e_{max} $ and then $a_\sigma(b_\sigma(\mu)) \leq \mu$.
Moreover $a_\sigma(b_\sigma(\mu)+ \e) > \mu$ for all sufficiently small $\e>0$, therefore the continuity of $a_\sigma$ allows to pass to the limit $\e \to 0$ and get
$a_\sigma(b_\sigma(\mu)) \geq \mu$.  We conclude that $a_\sigma(b_\sigma(\mu)) = \mu$ for all  $\mu \in ]0, \mbox{meas}(\Omega)[.$  It is straightforward to see that this also holds for $\mu=0$ and for $\mu= \mbox{meas}(\Omega).$ This ends the proof of \fref{inv-droite}.  The claim \fref{equiv-btheta} is now a direct consequence of definition \fref{btheta} and of \fref{inv-droite}. Indeed if $e<e_{max}$ and $\mu\geq 0$ are such that  $a_\sigma(e)\leq \mu$ then from the definition of $b_\sigma$ we get $e\leq b_\sigma(\mu)$. Conversly if $e\leq b_\sigma(\mu)$, then we just compose by the nondecreasing function  $a_\sigma$ and use \fref{inv-droite} to get $a_\sigma(e)\leq \mu$. 
This ends the proof of assertion {\em ii)}.

We now prove assertion {\em iii)}. Indeed,  assertion \fref{alkalam} is clearly true when the set  $\{e< e_{max} : a_\sigma(e)= \mu \}$ is  empty. Assume now that this set is not empty. Since $a_\sigma$ is nondecreasing and continuous, the set $\{ e< e_{max} : a_\sigma(e)= \mu\}$ is either reduced to one point or is an interval $I$ of $\RR$ with a nonempty interior. In the first case, claim \fref{alkalam} comes from the fact that the level sets of $\sigma$ are of zero measure.  In the second case, let  $[e_1,e_2]$ any subinterval of $I$ with $e_1<e_2$, we have
$$\begin{array}{ll}\mbox{meas}\left( \sigma^{-1}([e_1,e_2])\right)&\leq \mbox{meas} \{x\in \Omega;\   e_1\leq \sigma(x)\leq e_2 \} \\ &= a_\sigma(e_2)-a_\sigma(e_1)=0.
\end{array}$$
This implies that $\mbox{meas}\left(\sigma^{-1}(I)\right)=0,$
and therefore claim \fref{alkalam} is proved.

We finally prove assertion {\em vi)}. We use the properties of $a_\sigma$ stated in assertions {\em i)-ii)-iii)} and  write 
from \fref{Btheta}
 \begin{align*}\ds B_{\sigma}(\mu)&=\ds  \int_{a_{\sigma}(\sigma(x))< \mu} \sigma(x)dx =  \int_{a_{\sigma}(\sigma(x))\leq  \mu} \sigma(x)dx=\int_{\sigma(x))\leq  b_{\sigma}(\mu)} \sigma(x)dx     \\ 
&= \ds  \int_{\sigma(x))\leq b_\sigma(\mu)} \int_{e_{min}}^{\sigma(x)} dt dx + \mu e_{min} = \int_{e_{min}}^{b_\sigma(\mu)} \left(\int_{t\leq \sigma(x)<b_\sigma(\mu)} dx \right) dt+ \mu e_{min}\\ 
&\ds =\int_{e_{min}}^{b_\sigma(\mu)} \left( \mu - a_{\sigma}(t)\right) dt  + \mu e_{min}.
\end{align*}
Therefore, to prove \fref{Btheta-simple}, it remains to show that
$$  
\int_{e_{min}}^{b_\sigma(\mu)} \left( \mu - a_{\sigma}(t)\right) dt  + \mu e_{min} = \int_0^\mu b_\sigma (s) ds.$$
This identity  can be obtained by simply observing that  the distributional derivatives of the two sides with respect to $\mu$ are equal (using property \fref{inv-droite}). Since the two sides coincide at $\mu=0$ (note that $b_\sigma(0) = e_{min}$), we conclude that this identity holds true.

This ends the proof of Lemma \ref{lem_atheta}.

\end{proof}

\subsection{Proof Proposition \ref{f*theta-equi}.}
\label{proof-prop1}
It is well known, from the definition of $f^\sharp$ and the right-continuity of $\mu_f$, that we have
\be
\label{zarby}f^\sharp(r)>t \Longleftrightarrow  \mu_f(t)>r, \qquad \forall t\geq 0,\  r\geq 0.
\ee
Although \fref{zarby} is standard,  we briefly sketch its  elementary proof for a sake of completeness. Using the definition of $f^\sharp$  we  first have
$$ f^\sharp(r)>t \implies r< \mu_f(t).$$
If we take $\e>0$ and $ r<\mu_f(t+\e)$, we see that for all $t'$ such that $\mu_f(t') \leq r$ we have
$\mu_f(t')<\mu_f(t+\e)$ and then $t'>t+\e$. This means  $f^\sharp(r)\geq t+\e>t$.Thus
\be
\label{kk1} r<\mu_f(t+\e)\implies  f^\sharp(r)>t, \quad \forall \e >0.
\ee
Take now $r<\mu_f(t)$, then from the right-continuity of $\mu_f$ we have $\mu_f(t+\e) \to \mu_f(t)$ when $\e \to 0, \e>0.$ Therefore there exists $\e_0>0$ such that
$r<\mu_f(t+\e_0) \leq \mu_f(t)$, but from \fref{kk1}  this implies  that $f^\sharp(r)>t$. Thus \fref{zarby} is proved.
Now using \fref{zarby}   we have for $t\geq 0$
\be
\label{ttt}
\begin{array}{lll}
\ds \{x\in \Omega; \  f^{*\sigma}(x)>t\} &=&\ds  \{x\in \Omega;\  f^\sharp(a_\sigma(\sigma(x)))>t \ \mbox{and} \  \sigma(x)<e_{max}\} \\
& = &\ds \{x\in \Omega;\  a_\sigma(\sigma(x))< \mu_f(t) \ \mbox{and} \  \sigma(x)<e_{max}\}.
\end{array}
\ee
 
In particular if $\mu_f(t)=0$ then $\mu_{f^{*\sigma}}(t)=0= \mu_f(t)$.  We now assume that $\mu_f(t)>0$ and use \fref{alkalam}  and \fref{equiv-btheta} to get
$$\begin{array}{ll}\mbox{meas} \{x\in \Omega; \  f^{*\sigma}(x)>t\}&=\mbox{meas} \{x\in \Omega;\  a_\sigma(\sigma(x))< \mu_f(t) \ \mbox{and} \  \sigma(x)<e_{max}\}  \\
&= \mbox{meas} \{x\in \Omega;\  a_\sigma(\sigma(x))\leq \mu_f(t) \ \mbox{and} \  \sigma(x)<e_{max}\}\\
&= \mbox{meas} \{x\in \Omega;\  \sigma(x)\leq b_\sigma (\mu_f(t)) \ \mbox{and} \  \sigma(x)<e_{max}\}\\
&=\mbox{meas} \{x\in \Omega;\  \sigma(x)<b_\sigma (\mu_f(t))\}\\
&= a_\sigma ( b_\sigma(\mu_f(t))),
\end{array}$$
where we used the definition of $a_\sigma$.  We now  use identity \fref{inv-droite} of Lemma \ref{lem_atheta} and end  
the proof of the first part of Proposition \ref{f*theta-equi}, that is $f^{*\sigma}$ is equimeasurable with $f$.

To complete the proof of Proposition \ref{f*theta-equi}, it remains to show that $f^{*\sigma}$ is the only nonnegative function in $L^1(\Omega) $ which is a non increasing function of $\sigma(x)$ and is equimeasurable with $f$.  Let  $g(x)= F(\sigma(x))$  such a function.  We shall prove that $g=f^{*\sigma}$.

{\em Case 1, $e_{max} = +\infty$:}   In this case $a_\sigma$ is  a continuous and nondecreasing function from $\RR$ to $\RR^+$.  For all  $t\in \RR$, the set 
$\{e\in \RR: F(e)>F(t)\}$ is either empty or is an interval which does not contain $t$, and since $F$ is non increasing, the  interior of this set is of the form $]-\infty, R(t)[$, with $R(t)\leq t$.  Therefore, for all $x\in \RR^d$, $$\{y\in \RR^d: F(\sigma(y))>F(\sigma(x)\} \subset \{y\in \RR^d: \sigma(y)\leq R(\sigma(x))\} \subset \{y\in \RR^d: \sigma(y)\leq \sigma(x)\}.$$
This means that $\mu_g(g(x)) \leq a_\sigma(\sigma(x)).$   But since $g$ is equimeasurable to $f$, we have $\mu_g=\mu_f$ and then $\mu_f(g(x)) \leq a_\sigma(\sigma(x)).$ We conclude from the definition of $f^\sharp$ that
$$ g(x) \geq f^\sharp\circ a_\sigma(\sigma(x))= f^{*\sigma}(x),\quad \forall x\in \RR^d.$$
Since $g$ is equimeasurable to $f$ which itself is equimeasurable to $f^{*\sigma}$, we have
$$ \int_{\RR^d} (g(x) -  f^{*\sigma}(x)) dx =0,$$ and then conclude that
$$ g(x) =  f^{*\sigma}(x), \quad \mbox{for} \ a. e.  \  x\in \RR^d.$$
This ends the proof of Proposition \ref{f*theta-equi} in this case.

{\em Case 2, $e_{max} < +\infty$ and $\mbox{meas}(\Omega)<+\infty$:}  by assumption we have $a_\sigma(e_{max})=\mbox{meas}(\Omega)$
and therefore $\mbox{meas}\{x\in \RR^d: \sigma(x)\geq e_{max}\}=0.$ This means that $\sigma(x)<e_{max}$ for $a. e. \ x \in \RR^d.$  In particular
$$f^{*\sigma}(x) :=f^{\sharp}\circ a_\sigma(\sigma(x))\mathds{1}_{\sigma(x)<e_{max}}= f^{\sharp}\circ a_\sigma(\sigma(x)).$$
The equimeasurability can now be derived following exactly the same lines as in the previous case 1.

{\em Case 3, $e_{max} < +\infty$ and $\mbox{meas}(\Omega)=+\infty$:} In this case we do not have $\sigma(x)<e_{max}$ for $a. e. \ x \in \RR^d,$  but we shall prove
that $F(e)=0$ for all $e\geq e_{max}.$ Using that $g=F\circ \sigma$ is in $L^1(\Omega)$. Indeed, we know from \fref{mufinite} that for all $e>0$
$$\mbox{meas}  \{x\in \RR^d: F(\sigma(x))> e\} \leq \frac{1}{e}\|g\|_{L^1}.$$
Using the pseudo-inverse $\tilde F$ of the non increasing function $F$, we get
$$\mbox{meas}  \{x\in \RR^d: \sigma(x)<\tilde F (e)\} \leq \frac{1}{e}\|g\|_{L^1}.$$
This means that $a_\sigma (\tilde F (e))$ is finite and therefore  $\tilde F (e)< e_{max}$ by assumption on $\sigma$. Consequently
$F(e_{max}) \leq e$ for all $e>0$ and then $F(e_{max})=0$. This implies that $F(e)=0$ for all $e\geq e_{max}$.
We conclude that
$$ g(x) = f^{*\sigma}(x)= 0 \quad \mbox{if}\quad \sigma(x)\geq e_{max}.$$
The case where $\sigma(x)< e_{max}$ can be done in exactly similar way than Case 1.

This ends the proof of Proposition \ref{f*theta-equi}.


\subsection{Proof of Theorem 1}
\label{proofth1}

We proceed in several steps

\medskip

{\em Step 1:  Convexity of $B_\sigma$}:

\medskip

In this step we shall prove assertion {\em (i)} , that is the convexity of the function $B_\sigma$ defined by \fref{Btheta} on $[0,\mbox{meas}(\Omega)[$.  From \fref{alkalam} and  \fref{equiv-btheta}, we first observe that
$$B_\sigma (\mu)= \int_{a_\sigma(\sigma(x))\leq  \mu} \sigma(x) dx=\int_{\sigma(x)\leq  b_\sigma(\mu)} \sigma(x) dx,$$
where $b_\sigma$ is the pseudo-inverse of $a_\sigma$ defined by \fref{btheta}.
Note that from assertion {\em i)} of Lemma \ref{lem_atheta}, we have:   $a_\sigma(\sigma(x))\leq  \mu<\mbox{meas}(\Omega)\implies \sigma(x) < e_{max}$.  Let  $0\leq \mu_1< \mu_2<\mbox{meas}(\Omega)$ and $\mu= \frac{\mu_1 +\mu_2}{2}$, we have
\begin{align*}
B_\sigma(\mu_1) + B_\sigma(\mu_2) - 2 B_\sigma(\mu) &=  B_\sigma(\mu_2) - B_\sigma(\mu) - (B_\sigma(\mu)  - B_\sigma(\mu_1))  \\
&= \int_{b_\sigma(\mu)< \sigma(x)\leq b_\sigma(\mu_2)} \sigma(x) dx - \int_{b_\sigma(\mu_1)< \sigma(x)\leq b_\sigma(\mu)} \sigma(x) dx.\\
&\geq b_\sigma(\mu) \left[\mbox{meas}A_2- \mbox{meas}A_1\right],
\end{align*}
where $$A_1= \{x\in \Omega: b_\sigma(\mu_1)< \sigma(x)\leq b_\sigma(\mu)\}\quad \mbox{and} \quad A_2= \{ x\in \Omega: b_\sigma(\mu)< \sigma(x)\leq b_\sigma(\mu_2) \}.$$
Observing that  $\mbox{meas}A_1= a_\sigma \circ b_\sigma (\mu_2) - a_\sigma \circ b_\sigma (\mu)= \mu_2-\mu= \frac{\mu_2 -\mu_1}{2}, $ and similarly
$\mbox{meas}A_2= a_\sigma \circ b_\sigma (\mu) - a_\sigma \circ b_\sigma (\mu_1)= \mu-\mu_1=\frac{\mu_2 -\mu_1}{2}, $  we deduce that $\mbox{meas}A_1=\mbox{meas}A_2.$ This gives  the convexity of  $B_\sigma$.

\bigskip

{\em Step 2:  Layer cake representation and two-sides rearrangements} 

\medskip

Let $f$ and $q$ be two nonnegative functions in $ L^1(\Omega)$,  and  let $q^{*\sigma}$ be the symmetrization of $q$ with respect to the  function $\sigma$ as given by definition \ref{def1}. We write
\be
\label{bath}\begin{array}{ll} \ds   \int_\Omega \sigma(x) (f(x)-q^{*\sigma}(x))dx & \ds = \int_{0}^{+\infty} \left( \int _{x\in\Omega; f(x)>t} \sigma(x)dx-  \int _{x\in\Omega; q^{*\sigma}(x)>t} \sigma(x) dx\right) dt \\
 &\ds = \int_{0}^{+\infty} \left( \int _{D_1(t)} \sigma(x)  dx -  \int _{D_2(t)} \sigma(x) dx\right) dt,
              \end{array} \ee
 where
  \be
 \label{D1D2}
 D_1(t)= \{x\in \Omega; q^{*\sigma}(x)\leq t< f(x)\}, \qquad  D_2(t)= \{x\in \Omega;f(x)\leq t<q^{*\sigma}(x)\}.
 \ee
 Let
\be
\label{alpha-beta}
\alpha(t) = \mbox{meas} D_1(t),  \qquad  \beta(t)= \mbox{meas}D_2(t).
\ee
For any given $t>0$, we now claim that 
\be
\label{new-equi1}\alpha(t)= \mbox{meas}\{x\in \Omega;   q^{*\sigma}(x)\leq t\ \mbox{and}\  a_\sigma(\sigma(x)) < \mu_q(t)+ \alpha(t)\}.
\ee
This claim will be obtained by rearranging the set $D_1(t)$ "from the outside", that is rearranging it into an external ring (located between tow level sets of $\sigma$) of the same measure. Indeed, if $q^{*\sigma}(x) > t $, which means $q^\sharp(a_\sigma(\sigma(x))) > t $, then from the definition of $q^\sharp$ and $\mu_q$ we get
$a_\sigma(\sigma(x) )< \mu_q(t)$. We  deduce that
$$\{x\in \Omega;  q^{*\sigma}(x) > t \} \subset \{x\in \Omega;  a_\sigma (\sigma(x)) < \mu_q(t)+ \alpha(t)\},$$
and then
$$\begin{array}{l} \mbox{meas}\{x\in \Omega;  q^{*\sigma}(x) \leq t \ \mbox{and}\  a_\sigma (\sigma(x)) <  \mu_q(t)+ \alpha(t)\}= \\ 
\hspace{3cm} \mbox{meas}\{x\in \Omega;  a_\sigma (\sigma(x)) < \mu_q(t)+ \alpha(t)\} - \mbox{meas}\{x\in \Omega;  q^{*\sigma}(x) > t \}.\end{array}$$
But from assertion {\em i)} of Lemma \ref{lem_atheta}  we have $a_\sigma (\sigma(x)) < \mu_q(t)+ \alpha(t)\implies 
\sigma(x) <e_{max}$. Therefore
using the equimeasurability of $q$ and $q^{*\sigma}$ and assertions {\em ii)} and {\em iii)} of Lemma \ref{lem_atheta}  we get 
$$\mbox{meas}\{x\in \Omega:  q^{*\sigma}(x) \leq t \ \mbox{and}\  a_\sigma (\sigma(x)) < \mu_q(t)+ \alpha(t)\}=a_\sigma \circ b_\sigma (\mu_q(t)+ \alpha(t))- \mu_q(t)= \alpha(t),$$
which proves the claim \fref{new-equi1}.
Now we show similarly that we have
\be
\label{new-equi2}\beta(t)= \mbox{meas}\{x\in \Omega;  q^{*\sigma}(x) > t \ \mbox{and}\  a_\sigma(\sigma(x)) \geq  \mu_q(t)- \beta(t) \},
\ee
(note that  $\beta(t)\leq \mu_q(t)$). This will be obtained by rearranging the set $D_2(t)$ by an internal ring of the same measure (inside rearrangement). Indeed we observe that
$$\{x\in \Omega;  a_\sigma (\sigma(x)) < \mu_q(t)- \beta(t)\}  \subset \{x\in \Omega;  q^{*\sigma}(x) > t \}.$$
Thus
$$ \begin{array}{l} \mu_q(t) - \mbox{meas}\{x\in \Omega;  a_\sigma (\sigma(x)) < \mu_q(t)- \beta(t)\} = \\ 
\hspace{3cm} \mbox{meas}\{x\in \Omega;  q^{*\sigma}(x) > t \ \mbox{and}\  a_\sigma(\sigma(x)) \geq  \mu_q(t)- \beta(t) \},\end{array}$$
which gives (similarly as above) 
$$\mbox{meas}\{x\in \Omega;  q^{*\sigma}(x) > t \ \mbox{and}\  a_\sigma(\sigma(x)) \geq  \mu_q(t)- \beta(t) \}=
\mu_q(t) - a_\sigma\circ b_\sigma(\mu_q(t)- \beta(t))= \beta(t).$$
This ends the proof of claim \fref{new-equi2}.

Now let us return back to formula \fref{bath} and estimate each of the two involved integral terms.
First, we claim that
\be
\label{monotonie1}\forall t>0, \ \ \ \  \int _{q^{*\sigma}\leq t<f} \sigma(x)  dx \geq \ \int _{q^{*\sigma}(x) \leq t , \  a_\sigma(\sigma(x)) <\mu_q(t)+ \alpha(t) } \sigma(x) dx.
\ee
To see this, we set $D_1(t)=\{x\in \Omega; q^{*\sigma}(x)\leq t<f(x)\} $, $\tilde D_1(t)= \{x\in \Omega; q^{*\sigma}(x) \leq t \  \mbox{and} \  a_\sigma(\sigma(x)) < \mu_q(t)+ \alpha(t) \}$, and write
$$\int _{D_1(t)}\sigma(x) dx =   \int _{\tilde D_1(t)} \sigma(x) dx + \int _{D_1(t)\backslash \tilde D_1(t)} \sigma(x) dx - \int _{\tilde D_1(t)\backslash D_1(t)} \sigma(x) dx.$$
Let  $x \in D_1(t)\backslash \tilde D_1(t)$ and $y \in \tilde D_1(t)\backslash D_1(t)$. We have $a_\sigma(\sigma(x)) \geq \mu_q(t)+ \alpha(t)$ and $a_\sigma(\sigma(y)) < \mu_q(t)+ \alpha(t)$. Thus $a_\sigma(\sigma(y)) <a_\sigma(\sigma(x))$, which  implies $\sigma(x)>\sigma(y).$
Therefore for all $y\in \tilde D_1(t)\backslash D_1(t)$
$$  \int _{D_1(t)\backslash \tilde D_1(t)} \sigma(x) dx\ds \geq \sigma(y) \mbox{meas}\left(D_1(t)\backslash \tilde D_1(t)\right),$$
which, integrated with respect to $y\in \tilde D_1(t)\backslash D_1(t)$, yields
$$\mbox{meas}\left( \tilde D_1(t)\backslash D_1(t)\right)\int _{D_1(t)\backslash \tilde D_1(t)} \sigma(x) dx \geq 
 \mbox{meas}\left(D_1(t)\backslash \tilde D_1(t)\right)\int _{\tilde D_1(t)\backslash  D_1(t)} \sigma(y) dy.$$
Now the equimeasurability property \fref{new-equi1}(outside rearrangement) ensures that
$\mbox{meas}D_1(t)=\mbox{meas}\tilde D_1(t)$, and therefore  $\mbox{meas}\left( \tilde D_1(t)\backslash D_1(t)\right)=
\mbox{meas}\left(D_1(t)\backslash \tilde D_1(t)\right).$  Thus, the ongoing inequality clearly implies the claim (\ref{monotonie1}).

As above, we  also claim that
\be
\label{monotonie2}\forall t>0, \ \ \ \  \int _{f\leq t<q^{*\sigma}} \sigma(x)  dx \leq \ \int _{a_\sigma(\sigma(x))\geq \mu_q(t) -\beta(t), \  q^{*\sigma}(x)>t } \sigma(x)  dx,
\ee
 the proof of which is  similar to that of (\ref{monotonie1}).
 
 Injecting inequalities (\ref{monotonie1}) and (\ref{monotonie2}) into (\ref{bath}) yields
 $$\begin{array}{l}\ds  \int \sigma(x) (f(x)-q^{*\sigma}(x))dx \geq  \\ \hspace{1.5cm}\ds  \int_{0}^{+\infty} dt \left(\int _{\scriptsize \begin{array}{l}q^{*\sigma}(x) \leq t,\\ a_\sigma(\sigma(x)) <\mu_q(t)+\alpha(t)\end{array}} \sigma(x)dx-  \int _{\scriptsize \begin{array}{l}q^{*\sigma}(x) > t,\\ a_\sigma(\sigma(x)) \geq \mu_q(t)-\beta(t)\end{array}} \sigma(x)dx\right).
 \end{array}$$
   Since     $ \mu_q(t) \leq s \Leftrightarrow q^\sharp(s) \leq t$, we get
   $$\begin{array}{l}\ds  \int \sigma(x) (f(x)-q^{*\sigma}(x))dx \geq  \\ \hspace{1.5cm}\ds  \int_{0}^{+\infty} dt \left(\int _{\scriptsize  \mu_q(t)\leq a_\sigma(\sigma(x)) <\mu_q(t)+\alpha(t)} \sigma(x)dx-  \int _{\scriptsize  
   \mu_q(t)-\beta(t) \leq a_\sigma(\sigma(x))<\mu_q(t)}\sigma(x)dx\right).
 \end{array}$$

In terms of the function $B_\sigma$ defined by \fref{Btheta}, this writes
\be
 \label{base0}
 \int \sigma(x) (f(x)-q^{*\sigma}(x))dx\geq \ds \int_{0}^{+\infty} \left[ B_\sigma(\mu_q(t) + \alpha(t)) + B_\sigma(\mu_q(t) - \beta(t)) - 2B_\sigma (\mu_q(t)) \right]dt.
 \ee
 
 \medskip
 
 {\em Step 3: convexity estimates:}
 
 \medskip

 We rewrite \fref{base0} in the following form
\be
 \label{base1}
 \begin{array}{ll}
\ds  \int \sigma(x) (f(x)-q^{*\sigma}(x))dx &\geq \ds \int_0^{\|q\|_{L^{\infty}}} \left[ B_\sigma(\mu_q(t) + \beta(t)) + B_\sigma(\mu_q(t) - \beta(t)) - 2B_\sigma (\mu_q(t)) \right]dt \\ 
&\hspace{1cm} + \ds  \int_{0}^{+\infty}  \left[ B_\sigma (\mu_q(t) + \alpha(t)) - B_\sigma (\mu_q(t) +\beta(t))  \right]dt,\\&=\ds 
T_1+ T_2.
 \end{array}
\ee  
where we have restricted the integration domain to $]0,{\|q\|_{L^{\infty}}}[$ in the first integral of the rhs term, since $\beta(t) =0$ for $t\geq \|q\|_{L^{\infty}}$.
Note that  $\|q\|_{L^{\infty}}$ may be infinite.
We start by estimating from below the term 
$$T_2=  \int_{0}^{+\infty}  \left[ B_\sigma (\mu_q(t) + \alpha(t)) - B_\sigma (\mu_q(t) +\beta(t))  \right]dt.$$ 

To this aim, we observe first that the convex function $B_\sigma$ satisfies
\be \label{bla-con} B_\sigma(\mu_1) -B_\sigma(\mu_2) \geq (\mu_1-\mu_2) b_\sigma (\mu_2), \ \ \forall 0\leq \mu_1, \mu_2 <\mbox{meas}(\Omega).
\ee
Indeed, assume $\mu_1\geq \mu_2$ and write (using in particular \fref{alkalam} and \fref{equiv-btheta}) 
$$\begin{array}{lll} \ds B_\sigma(\mu_1) -B_\sigma(\mu_2) &= &\ds \int_{\mu_2< a_\sigma(\sigma(x))\leq \mu_1 } \sigma(x) dx\\
&=& \ds \int_{b_\sigma(\mu_2)< \sigma(x)\leq b_\sigma(\mu_1) } \sigma(x) dx \\
&\geq &\ds  b_\sigma(\mu_2) (a_\sigma(b_\sigma(\mu_1) )- a_\sigma(b_\sigma(\mu_2) )= b_\sigma(\mu_2) \mu_1-\mu_2).
\end{array}$$
and similarly the same holds if $\mu_1\leq \mu_2$.  Using \fref{bla-con} we then get
$$T_2\geq\int_{0}^{+\infty} (\alpha(t)-\beta(t))  b_\sigma(\mu_q(t)+\beta(t))dt.$$
Now we observe that 
\be
\label{alpha-beta*}
\beta(t) -\alpha(t)= \mu_q(t) - \mu_f(t) = \mu_{q^*}(t) - \mu_{f^*}(t) = \beta_*(t)-\alpha_*(t),
\ee
with
$$
\alpha_*(t) = \mbox{meas}\{x\in\Omega; q^*(x)\leq t< f^*(x)\},  \qquad  \beta_*(t)= \mbox{meas}\{x\in\Omega; f^*(x)\leq t<q^*(x)\}.
$$
Therefore, using the monotonicity of $b_\sigma$ and the fact that $0\leq \beta(t)\leq \mu_q(t)$ , we get
$$T_2\geq  \int_{0}^{+\infty}(\alpha_*(t)-\beta_*(t)) b_\sigma(\mu_q(t)+\beta(t))dt
 \geq \int_{0}^{+\infty}\left[b_\sigma(\mu_q(t))\alpha_*(t) - b_\sigma(2\mu_q(t))\beta_*(t)\right] dt$$
In particular inequality \fref{ineq11} is proved.
To end the proof of inequality \fref{ineq1}, it remains
to estimate from below the term 
$$T_1=  \int_0^{\|q\|_{L^{\infty}}} \left[ B_\sigma(\mu_q(t) + \beta(t)) + B_\sigma(\mu_q(t) - \beta(t)) - 2B_\sigma (\mu_q(t)) \right]dt.$$
Using the definition of the function $H_\sigma(\mu)$ given in Theorem \ref{thm1} and observing that $\beta(t)\leq \mu_q(t) $, we get
$$T_1 \geq  \ds \int_{0}^{{\|q\|_{L^{\infty}}} } H_\sigma(\mu_q(t) )\beta(t)^2 dt $$
By Cauchy-Schwarz (or Jensen) inequality we then get 
 \be
 \label{base2}
T_1\geq   \left({\int_0^{\|q\|_{L^{\infty}}} \frac{dt}{H_\sigma(\mu_q(t) )}} \right)^{-1}\left(  \int_{0}^{+\infty}\beta(t) dt\right)^2.
\ee 
Now to estimate $ \int_0^{+\infty} \beta(t) dt $, we first observe that
$$\int_0^{+\infty} \beta(t)dt =\int_{\Omega}(f(x) - q^{*\sigma} (x))_+ dt,  \quad\mbox{and} \quad \int_0^{\infty} \alpha(t)dt =\int_{\Omega}  (q^{*\sigma}(x) -f(x))_+ dx,$$
which implies 
\be
\label{beta-p-alpha}\int_0^{\infty} (\alpha(t)+\beta(t))dt =\int_{\Omega} 
\left( (q^{*\sigma}(x) -f(x))_+ +(f(x) - q^{*\sigma} (x))_+\right) dx= \|f-q^{*\sigma} \|_{L^1}.
\ee
Note also that from \fref{alpha-beta*} we have
\be
\label{beta-m-alpha}
\ds \int_0^{\infty} (\beta(t)-\alpha(t))dt \ds =\int_0^{\infty} (\beta_*(t)-\alpha_*(t))dt   \ds = \int \left(q^*(x)-f^* (x)\right)dx = \|q\|_{L^1}-\|f\|_{L^1}.
\ee
Summing \fref{beta-p-alpha}and  \fref{beta-m-alpha}  we get
$$
\int_0^{+\infty}\beta(t)dt  = \frac{1}{2} \left(  \|f-q^{*\sigma} \|_{L^1} +\|q\|_{L^1}-\|f\|_{L^1} \right),
$$
and use  identities \fref{beta-p-alpha} and  \fref{beta-m-alpha} to estimate $ \int_0^{+\infty} \beta(t) dt $.
We now report the obtained estimate into \fref{base2} to get
 \be
 \label{base3}
T_1 \geq  \left(4 \int_0^{\|q\|_{L^{\infty}}} \frac{dt}{H_\sigma(\mu_q(t) )}\right)^{-1}  \left( \|f-q^{*\sigma} \|_{L^1} +\|q\|_{L^1}-\|f\|_{L^1} \right)^2 .
\ee 
Reporting \fref{base2} and \fref{base3}  into \fref{base1} ends the proof of  the inequality \fref{ineq1} in Theorem \ref{thm1}. Finally, inequality \fref{ineq2}
is directly obtained from \fref{ineq1}  by simply taking $q=f$.

\subsection{Proof of Corollary \ref{cor1}}
\label{proof-cor1}
To prove Corollary \ref{cor1}, we just apply theorem 1 with $\sigma(x)=|x|^m, \ \forall x\in\RR^d.$
In this case we have
$$a_\sigma(s) =K_d s^{d/m},  \quad b_\sigma (s)=a_\sigma^{-1}(s)= K_d^{-m/d} s^{m/d}, \quad   B_\sigma(s)= \frac{d}{m+d} K_d^{-m/d}s^{1+m/d} .$$
An elementary computation shows that
$$\begin{array}{ll} H_\sigma(\mu) &\ds = \inf_{0<s<\mu} \frac{B_\sigma (\mu +s)+ B_\sigma(\mu-s)-2B(\mu)}{s^2} \\
&\ds =\frac{d}{m+d} K_d^{-m/d}  \inf_{0<s<\mu}  \frac{(\mu +s)^{1+m/d}+ (\mu-s)^{1+m/d} -2\mu^{1+m/d}}{s^2}\\
&\ds =\frac{d}{m+d} K_d^{-m/d} \mu^ {-1+m/d} \inf_{0<s<1}    \frac{(1 +s)^{1+m/d}+ (1-s)^{1+m/d} -2}{s^2}\\\
\end{array}$$
An elementary analysis of the function $s\mapsto (1 +s)^{1+m/d}+ (1-s)^{1+m/d} -2 - \frac{m}{d} \left( 1+ \frac{m}{d}\right) s^2$
reveals that (for $0\leq m\leq d$)
$$  \inf_{0<s<1}    \frac{(1 +s)^{1+m/d}+ (1-s)^{1+m/d} -2}{s^2}=  \frac{m}{d} \left( 1+ \frac{m}{d}\right),$$
and then 
$$ H_\sigma(\mu) =\frac{m}{d} K_d^{-m/d} \mu^ {m/d-1}.$$
 Thus
$$ \int_0^{+\infty} \frac{1}{H_\sigma(\mu_f(s))} ds =\frac{d}{m} K_d^{m/d} \int_0^{\|f\|_{L^\infty}}\mu_f(s)^ {1-m/d}ds,  $$
where we have used that$\mu_f(s)=0$ for $s>\|f\|_{L^\infty}$.
Now, we use the concavity of the function $\mu\rightarrow \mu^ {1-m/d}$ (since $0\leq m\leq d$) and get by Jensen inequality

$$ \int_0^{+\infty} \frac{1}{H_\sigma(\mu_f(s))} ds \leq \frac{d}{m} K_d^{m/d} \|f\|_{L^\infty}^{m/d} \left(\int_0^{\|f\|_{L^\infty}}\mu_f(s)ds \right)^ {1-m/d}.$$
Observing that $\ds \int_0^{\|f\|_{L^\infty}}\mu_f(s)ds = \|f\|_{L^1}$, we obtain
$$ \int_0^{+\infty} \frac{1}{H_\sigma(\mu_f(s))} ds \leq \frac{d}{m} K_d^{m/d} \|f\|_{L^\infty}^{m/d}\|f\|_{L^1}^{1-m/d}.$$
Reporting this into inequality \fref{ineq2} of Theorem \ref{thm1}  finally yields the desired inequality \fref{bathtub-classic}.
This ends the proof of Corollary \ref{cor1}.

\section{Quantitative stability inequalities for Valsov-Poisson and 2D Euler systems}
\subsection{Stability inequality for the Vlasov-Poisson system: proof of Theorem \ref{thm-VPG}}
\label{sectionVPP}
This section is devoted to the proof of Theorem \ref{thm-VPG}. We shall apply Theorem \ref{thm1} with  $\Omega=\RR^d, d=6.$   Let $f_0$  be a non-zero compactly supported and radial steady state of the Vlasov-Poisson system, which is a decreasing  function of the microscopic energy:
\be
\label{theta-vlasov}f_0(x,v)= F(e_{0}(x,v)),  \qquad e_{0}(x,v)= |v|^2/2 + \phi_{f_0}(x) - \phi_{f_0}(0),
\ee
where $F$ is a a continuous and decreasing function. The function $e_0$ will play the role of the function $\sigma$ used in Theorem \ref{thm1}. Since $\phi_{f_0} $ is radial and satisfies the Poisson equation,  it is known that  $\phi_{f_0}$ is a negative, continuous, and radially increasing function. This implies that $\phi_{f_0}(x) \geq \phi_{f_0}(0)$, and therefore the function $e_0$ is nonnegative.
To apply Theorem \ref{thm1}, we need to express the Jacobien $a_{e_0}$ and the generalized symmetrization
$f^{*e_{0}}$. These expressions have been established and used in \cite{LMR}. We have

$$a_{e_0}(s)= \frac{8\pi \sqrt{2}}{3} \int_{\RR^3}  \left( s+\phi_{f_0}(0) -\phi_{f_0}(x)\right)_+^{3/2} dx.$$
$$f^{*e_{0}}(x,v) = f^{\sharp} \circ a_{e_0}(e_0(x,v)),$$
where we recall that $f^\sharp$ is the pseudo-inverse of $\mu_f$.
From \cite{LMR}, we know that $a_{e_0}$ is a $C^1$-diffeomorphism from $[0, -\phi_{f_0}(0)[$ to $[0, +\infty[$, that $a_ {e_0}(s)=0$ for all $s\leq0$, and that $a_{e_0}(0)=a_{e_0}'(0)=0$ (note that here the definition of $a_ {e_0}$ only differs from that given in  \cite{LMR} by a constant). Here we have $e_{min}=0, e_{max}= -\phi_{f_0}(0)$ and it is easy to check that
$$ \mbox{meas}\{ (x,v)\in \RR^6:  e_0(x,v)= \lambda\}=0, \quad \forall \lambda < -\phi_{f_0}(0).$$ 
Now, from Lemma \ref{lem_atheta}, we know that the function $B_{e_0}(\mu)$  given by \fref{Btheta} is well defined and has the following expression. 
\be 
B_\sigma(s)= \int_0^\mu a_{e_0}^{-1}(s) ds
\ee
Consequently, we are allowed to apply inequality \fref{ineq1} of Theorem \ref{thm1}, with $f \in{\mathcal E}$ and $q=f_0$. 
Noting that $f_0^{*e_0} = f_0$  and that
$$ \|f\|_{L^1} - \|f_0\|_{L^1} = \|f^*\|_{L^1}  -    \|f_0^*\|_{L^1} \leq \| f^* -    f_0^*\|_{L^1}\leq  \| f -    f_0\|_{L^1},$$ we get
\be
\label{inequalityH1}
\begin{array}{ll}
\left(\|f-f_0\|_{L^1}- \|f^*-f_0^*\|_{L^1}\right)^2    \leq &\ds  K(f_0^*, \phi_{f_0})\left[  \ \int_{\RR^6}  \left(|v|^2/2 + \phi_{f_0}(x) - \phi_{f_0}(0)\right) (f- f_0) dx dv +\right. \\ &\ds \left.  \hspace{2cm}+ \int_{0}^{+\infty} a_{e_0}^{-1}(2\mu_{f_0} (s)) \beta_{f^*,f_0^*}(s)ds \right],
  \end{array} 
\ee
where $\beta_{f^*,f_0^*}(s) $ is defined by \fref{betafg} and where the constant $K(f_0^*, \phi_{f_0})$ (which will be shown to be finite below) is given by \fref{Kq*} with $\mu_{f_0}(s)= \mbox{meas}\{(x,v)\in\RR^6; f_0(x,v)>s\}$. 
Since  $0\leq a_{e_0}^{-1}(s)\leq -\phi_{f_0}(0), \ \forall s \geq0$, and 
$$ \int_{0}^{+\infty} \beta_{f^*,f_0^*}(s)ds =\int (f_0^*-f^*)_+ dxdv \leq \|f^*-f_0^*\|_{L^1},$$ we deduce that
inequality \fref{inequalityH1} implies
\be
\label{inequalityH2}
\begin{array}{ll}
\left(\|f-f_0\|_{L^1}- \|f^*-f_0^*\|_{L^1}\right)^2    & \leq   \ds    K(f_0^*, \phi_{f_0})\ds  \left[ \int_{\RR^6}  \left(|v|^2/2 + \phi_{f_0}(x)\right) (f- f_0) dx dv\right.  \\
&\hspace{4cm}\left. \ds +{\color{white} \int } 2 |\phi_{f_0}(0)|\|f^*-f_0^*\|_{L^1}\right] \\
 \end{array}
\ee
where we have used that $\int (f-f_0) dxdv = \int (f^*-f_0^*) dxdv \leq \|f^*-f_0^*\|_{L^1}$.

 Now we shall prove that $ K(f_0^*, \phi_{f_0})$ is finite. We recall that
$$K(f_0^*, \phi_{f_0}) =4 \int_0^{\|f_0\|_{L^{\infty}}} \frac{dt}{H_{e_0}(\mu_{f_0}(t))} ,
 $$
 where
 $$H_{e_0}(\mu) = \inf_{0<s\leq\mu}  \frac{B_{e_0}(\mu +s)+ B_{e_0}(\mu-s)-2B_{e_0}(\mu)}{s^2}, \quad \forall \mu>0,$$
 and where $B_{e_0}(s) = \int_0^s a_{e_0}^{-1}(\tau) d\tau.$ We observe that  $a_{e_0}$ is an increasing convex function and then 
$a_{e_0}'\circ a_{e_0}^{-1}(\mu_{f_0}(t)-s)\leq  a_{e_0}'\circ a_{e_0}^{-1}(\mu_{f_0}(t)) \leq a_{e_0}'\circ a_{e_0}^{-1}(\mu_{f_0}(0))= a_{e_0}'\circ a_{e_0}^{-1}(\mbox{meas}(\mbox{Supp}(f_0)))$. Therefore, using a Taylor expansion
$$\begin{array}{ll}\ds H_{e_0}(\mu_{f_0}(t)) &=  \ds \inf_{0<s\leq\mu_{f_0}(t)}  \frac{B_{e_0}(\mu_{f_0}(t) +s)+ B_{e_0}(\mu_{f_0}(t)-s)-2B_{e_0}(\mu_{f_0}(t))}{s^2} \\ & \ds = \inf_{0<s\leq\mu_{f_0}(t)}\int_0^1 (1-\lambda ) \left( \left(a_{e_0}^{-1}\right)' (\mu_{f_0}(t) + \lambda s) 
+\left(a_{e_0}^{-1}\right)' (\mu_{f_0}(t) - \lambda s)\right) d\lambda,\\ & \ds  \geq  \frac{1}{ 2 a_{e_0}'\circ a_{e_0}^{-1}(\mbox{meas}(\mbox{Supp}(f_0))},
\end{array}$$
 and then
$$K(f_0^*, \phi_{f_0}) \leq 8\, \|f_0\|_{L^{\infty}}\, a_{e_0}'\circ a_{e_0}^{-1}(\mbox{meas}(\mbox{Supp}(f_0)) <\infty. $$
We then report the following elementary identity
$$ {\mathcal H}(f) - {\mathcal H}(f_0) = \int_{\RR^6}  \left(\frac{|v|^2}{2} + \phi_{f_0}(x)\right) (f- f_0) dx dv -\frac{1}{2} \|\nabla \phi_f -\nabla \phi_{f_0}\|_{L^2}^2$$
into \fref{inequalityH2}, and get 
\be
\label{inequalityH3}
\begin{array}{ll}\left(\|f-f_0\|_{L^1}- \|f^*-f_0^*\|_{L^1}\right)^2   \leq   \ds    K(f_0^*, \phi_{f_0})& \ds  \left[ {\mathcal H}(f) - {\mathcal H}(f_0) +  \frac{1}{2} \|\nabla \phi_f -\nabla \phi_{f_0}\|_{L^2}^2  \right. \\  &\ds  \left.
\hspace{1.5cm} {\color{white} \frac{1}{2}}+2 |\phi_{f_0}(0)\|f^*-f_0^*\|_{L^1}\right] 
\end{array}
\ee
 which directly implies the global control \fref{ineq-grav-glob} in Theorem \ref{thm1}. So far we have obtained a {\em global control} of  $\|f-f_0\|_{L^1}$ by the relative energy ${\mathcal H}(f) - {\mathcal H}(f_0) $, the relative rearrangements $ \|f^*-f_0^*\|_{L^1} $ and the relative potential
$\|\nabla \phi_f -\nabla \phi_{f_0}\|_{L^2}^2$.   Now we shall prove the local control
\fref{ineq-grav-loc}. The control  of $\|\nabla \phi_f -\nabla \phi_{f_0}\|_{L^2}^2$ by the relative Hamiltonian has been obtained locally, but quantitatively, in  \cite{LMR} (Proposition 3.1) for perturbations $f$ which are  equimeasurable to $f_0$. Here we just refer to \cite{LMR} for  more details on this control and give the argument that allows to extend this control (in a quantitative way) to perturbations which are not necessarily equimeasurable with the initial data. We first write
(using the equimeasurability of $f$ and $f^{*e_\phi}$)

\be
\label{z1}
\begin{array}{ll}
\ds {\mathcal H}(f) - {\mathcal H}(f_0)&\ds = \int_{\RR^6}  \left(\frac{|v|^2}{2} + \phi(x)\right) (f- f_0) dx dv +\frac{1}{2} \|\nabla \phi -\nabla \phi_{f_0}\|_{L^2}^2\\
& \ds = \int_{\RR^6}  e_\phi (x,v)(f- f^{*e_\phi}) dx dv  + \int_{\RR^6}  (e_\phi(x,v)-\|\phi\|_{L^\infty}) (f^{*e_\phi}- f_0^{*e_\phi}) dx dv \\
&\ds 
+\int_{\RR^6}  e_\phi(x,v) (f_0^{*e_\phi}- f_0) dx dv+ \frac{1}{2} \|\nabla \phi -\nabla \phi_{f_0}\|_{L^2}^2,
\end{array}
\ee
where, for shortening, we have denoted $\phi=\phi_f$ and where
$$ e_\phi(x,v) =  |v|^2/2 + \phi(x) + \|\phi\|_{L^\infty} ,\qquad a_{e_\phi}(s)= \frac{8\pi \sqrt{2}}{3} \int_{\RR^3}  \left( s-\|\phi\|_{L^\infty} -\phi(x)\right)_+^{3/2} dx, \quad \mbox{and}$$
$$f^{*e_{\phi}}(x,v) = f^{\sharp} \circ a_{e_\phi}(e_\phi(x,v)).$$
We recall that $a_{e_\phi}(s)= \mbox{meas}\{ (x,v)\in \RR^6, e_\phi(x,v) <s\}$, and that from \cite{LMR}, $a_{e_\phi}$ is a continuous  and strictly increasing function from $[0,\|\phi\|_{L^\infty}[$ to
$[0,+\infty[$.
We have
$$\begin{array}{ll}\ds  \int_{\RR^6}  e_\phi(x,v) f^{*e_\phi}dx dv &\ds = \int_{\RR^6}  \left(\int_0^{f^{*e_\phi}(x,v)} dt\right) 
e_\phi(x,v)dx dv \\ &=\ds \int_0^{+\infty}\left( \int_{f^{*e_\phi}>t} e_\phi(x,v)dx dv\right) dt.
\end{array}$$
Since from \fref{zarby} $\{(x,v): f^{*e_\phi}(x,v)>t\}=\{(x,v): e_\phi(x,v)<a_{e_\phi}^{-1}(\mu_f(t))\}$, we have
$$\begin{array}{ll}\ds  \int_{\RR^6}  e_\phi(x,v) f^{*e_\phi}e_\phi(x,v)dx dv &=\ds \int_0^{+\infty}\left( 
\int_{0\leq e_\phi<a_{e_\phi}^{-1}(\mu_f(t))}e_\phi(x,v) dx dv\right) dt \\
&\ds  = \int_0^{+\infty} B_{e_\phi}(\mu_f(t)) dt = \int_0^{+\infty}  \left( \int_{0}^{\mu_f(t)}a_{e_\phi}^{-1}(s) ds \right) dt \\
&\ds = \int_0^{+\infty}  \left( \int_{\mu_f(t)>s} dt  \right) a_{e_\phi}^{-1}(s) ds\\
&\ds = \int_0^{+\infty}  f^\sharp (s)  a_{e_\phi}^{-1}(s) ds,
\end{array}$$
where we have used identity \fref{Btheta-simple} of Lemma \ref{lem_atheta}, and  the fact that  $\mu_f(t)>s\Leftrightarrow t<f^\sharp (s), \forall t,s \geq 0, $. We then deduce that
$$\begin{array}{ll}\left | \ds  \int_{\RR^6}  (e_\phi(x,v)-\|\phi\|_{L^\infty}) (f^{*e_\phi}- f_0^{*e_\phi}) dx dv \right |&\ds =\left | \int_0^{+\infty}  \left( f^\sharp (s) - f_0^\sharp (s) \right)(a_{e_\phi}^{-1}(s)-\|\phi\|_{L^\infty}) ds\right |\\ & \ds\leq  
\|\phi\|_{L^\infty} \|f^\sharp  - f_0^\sharp\|_{L^1(\RR)}=  \|\phi\|_{L^\infty} \|f^*  - f_0^*\|_{L^1(\RR^d)},
\end{array}$$
where we have used $a_{e_\phi}^{-1}(s)\leq\|\phi\|_{L^\infty}$.
We now report this estimate  in \fref{z1} to get
\be
\label{z2}
 {\mathcal H}(f) - {\mathcal H}(f_0)\geq - \|\phi\|_{L^\infty} \|f^*  - f_0^*\|_{L^1} + {\mathcal J} (\phi)-{\mathcal J} (\phi_{f_0}),
\ee
where
\be \label{expJ} {\mathcal J} (\phi)= \int_{\RR^6}  e_\phi(x,v) f_0^{*e_\phi}dx dv+ \frac{1}{2} \|\nabla \phi\|_{L^2}^2,
\ee
and where we have used the inequality 
 $\int_{\RR^6}  e_\phi (x,v)(f- f^{*e_\phi}) dx dv\geq 0$ which is a direct consequence of Theorem \ref{thm1}.
Finally, we combine this estimate \fref{z2} with the result of Proposition 3.1 in \cite{LMR}, and use assumption \fref{local-assumption} to get the local estimate \fref{ineq-grav-loc} of Theorem \ref{thm-VPG}.

\subsection{Stability inequalities for 2D Euler equation}
\subsubsection{Proof of Theorem \ref{thm-Euler1}}
\label{sectionEuler1P}
- {\em Proof of the stability inequality \fref{ineq-E2D-radial}}.   We will apply Theorem \ref{thm1} with 
$$\Omega = B(0,R), \ \mbox{and} \  \sigma(x)=|x|^2, \quad \forall x\in \Omega.$$
We have
$$a_\sigma (e) = \mbox{meas}\{x\in B(0,R); |x|^2< e\}= \pi \mbox{min}(e,R^2), \quad \forall e\geq 0.$$
The right inverse $b_\sigma$ is then defined on $[0,\pi R^2[$ by
$b_\sigma(s)= \frac{s}{\pi}.$ This gives
$$B_\sigma (s) = \int_0^s b_\sigma (\tau) d\tau= \frac{s^2}{2 \pi}, \mbox{and}$$
$$H_\sigma(\mu)= \inf_{0<s\leq \mu} \frac{B_\sigma (\mu+s)+B_\sigma (\mu-s)-2B_\sigma (\mu)}{s^2}=\frac{1}{\pi}.$$
We then obtain that the constant \fref{Kq*} is indeed finite for all $q\in L^\infty$:
$$K(q^*, \sigma)= 4\int_0^{\|q\|_{L^\infty}}  \frac{dt}{H_\sigma(\mu_q(t))}= 4\pi \|q\|_{L^\infty}. $$
We now consider a steady state $q$ of the form \fref{steady-radial}
and apply inequality \fref{ineq1} to a weak solution $\omega(t)$ to the 
2D Euler equation:
$$  \begin{array}{ll}
\ds \left( \|\omega(t)-q\|_{L^1(\Omega)}+\|q\|_{L^1}-\|\omega(t)\|_{L^1} \right)^2 \leq &\ds  4\pi{\|q\|_{L^\infty}} \left[ \int_{\Omega} |x|^2(\omega(t,x)-q (x)) dx\right. \\ 
\ds &\ds \hspace{1.5cm} \left.  +\frac{2}{\pi} \int_0^{+\infty} \mu_q(s) \beta_{\omega_{in}^*,q^*}(s) ds\right],
\end{array}
$$
Note that here $q=q^*=q^{*\sigma}.$ Inequality \fref{ineq-E2D-radial} is therefore a direct consequence
of the ongoing inequality and of the conservation properties of the solution
$$\omega(t)^*= \omega_{in}^*, \  \mbox{and}\  \int_{\Omega} |x|^2\omega(t,x)dx \leq \int_{\Omega} |x|^2\omega_{in}(x)dx,$$
where $\omega_{in}=\omega(0,x).$
Now to get \fref{ineq-E2D-radial-c}  from \fref{ineq-E2D-radial}  in case of comapctly supported steady states $q$, we just observe
that $\mu_q(s)\leq \mbox{meas}(\mbox{Supp}(q))$ and that
$\int_0^{+\infty} \beta_{\omega^*,q^*}(s) ds \leq \|\omega^*-q^*\|_{L^1(\Omega)}.$

- {\em Proof of the stability inequality \fref{ineq-E2D-rectangular}}.  
Consider now the following domain
$$\Omega=]0,L_1[\times ]0,L_2[, \quad \sigma (x)=x_2, \quad \forall x=(x_1,x_2)\in \Omega.$$
We have
$$a_\sigma (e) = \mbox{meas}\{x\in \Omega; x_2< e\}= L_1 \mbox{min}(e,L_2), \quad \forall e\geq 0.$$
The right inverse $b_\sigma$ is then defined on $[0,L_1L_2[$ by
$b_\sigma(s)= \frac{s}{L_1}.$ This gives
$$B_\sigma (s) = \int_0^s b_\sigma (\tau) d\tau= \frac{s^2}{2 L_1}, \quad \mbox{and}$$
$$H_\sigma(\mu)= \inf_{0<s\leq \mu} \frac{B_\sigma (\mu+s)+B_\sigma (\mu-s)-2B_\sigma (\mu)}{s^2}=\frac{1}{L_1}.$$
We then check that the constant \fref{Kq*} is indeed finite for all $q\in L^\infty$:
$$K(q^*, \sigma)= 4\int_0^{\|q\|_{L^\infty}}  \frac{dt}{H_\sigma(\mu_q(t))}= 4L_1 \|q\|_{L^\infty}. $$
If $q$ is a steady state  of the form \fref{ss-x2}, then we apply inequality \fref{ineq1} to a weak solution $\omega(t)$ to the 
2D Euler equation and get:
$$
\begin{array}{ll}
\ds \left( \|\omega(t)-q\|_{L^1(\Omega)}+\|q\|_{L^1}-\|\omega(t)\|_{L^1} \right)^2 \leq &\ds  4L_1{\|q\|_{L^\infty}} \left[ \int_{\Omega} x_2(\omega(t,x)-q (x)) dx\right. \\ 
\ds &\ds \hspace{1.5cm} \left.  +\frac{2}{L_1}  \int_0^{+\infty} \mu_q(s) \beta_{\omega_{in}^*,q^*}(s) ds\right],
\end{array}
$$
and again inequality \fref{ineq-E2D-rectangular} is a direct consequence of this inequality and of the conservation properties of the solution $\omega(t)$. The case of compactly supported steady state $q$ is also
obtained directly from \fref{ineq-E2D-rectangular} as in the radial case, using $\mu_q(s)\leq \mbox{meas}(\mbox{Supp}(q))$ and 
$\int_0^{+\infty} \beta_{\omega^*,q^*}(s) ds \leq \|\omega^*-q^*\|_{L^1(\Omega)}.$.

\subsubsection{Proof of Theorem \ref{thm-Euler2}}
\label{proof-Euler2}
To prove Theorem \ref{thm-Euler2}, we first observe that
\be
\label{alhadra1} {\mathcal H}(\omega) -{\mathcal H}(\omega_0) = \int_\Omega \psi_0 (\omega-\omega_0) dx +  \frac{1}{2} \int_\Omega |\nabla \psi -\nabla\psi_0(x)|^2 dx,
\ee
where the Hamiltonian $\mathcal H$ is given by \fref{energy-Ham-euler}. We recall that $\psi$ and $\psi_0$ are the stream functions associated 
to $\omega$ and $\omega_0$ respectively, and are nonnegative functions (since $\omega$ and $\omega_0$ are nonnegative). 
 To control the first term of the rhs of equality \fref{alhadra1}, we shall use Theorem \ref{thm1} with $\sigma(x)= \psi_0(x).$ In order to show  that the assumptions of Theorem \ref{thm1}  on $\sigma$ are satisfied, we need the following lemma.
 
 \begin{Lemma}[Level sets of $\psi_0$ are of zero measure]
 \label{lem2}
 Let $F$ be a nonnegative, nonincreasing, and continuous function on $\RR$ such that $F(0) \ne 0.$  Assume that $\psi_0$  satisfies
  $$\Delta \psi_0=-\omega_0 (x)= -F(\psi_0(x))   \ \ on \ \ \Omega, \quad \mbox{and} \ \quad   \psi_0=0 \ \mbox{on}\  \pa \Omega.$$
  Then the level sets of $\psi_0$ are of zero measure: 
  $$\mbox{meas}\{x\in \Omega; \psi_0(x) =\lambda \} =0, \qquad \forall \lambda \in \RR.$$
 \end{Lemma}
 \begin{proof}
Assume that there exists $\lambda\in \RR$ such that $\mbox{meas}\{x\in \Omega; \psi_0(x) =\lambda \} >0$. We know from the maximum principle that  $\psi_0$ is nonnegative and therefore $\lambda\geq 0$. From \cite{HMP} (page 67), we also know that $\nabla \psi_0=0$ almost  everywhere on $\{x\in \Omega; \psi_0(x) =\lambda \}$ and recursively $\Delta \psi_0 =0$ almost everywhere on the set $\{x\in \Omega; \nabla\psi_0(x) =0\}$. Therefore $\Delta \psi_0 =0$ almost everywhere on $\{x\in \Omega; \psi_0(x) =\lambda \}$ and then
$F(\lambda)=0$. Since $F$ is non increasing and  $\lambda \geq0$, we get $F(0)=0$ and this contradicts  our assumption on $F$.
 \end{proof}
 We shall first use this lemma to prove the first assertion {\em i)} of Theorem \ref{thm-Euler2}.  Since the level sets of $\psi_0$ are of zero measure, it is clear that
 $Supp(\psi_0) = \Omega$ and that $\mu_0(0)= \mbox{meas}(\Omega)$.  We also deduce from this property that the distribution function $\mu_{0}$ of $\psi_0$ is a continuous and non increasing function from $[0,\|\psi_0\|_{\infty}[$ to $]0, \mbox{meas}(\Omega)]$.
 This implies that $\psi_0^{\sharp}$ is strictly decreasing on 
 $]0, \mbox{meas}(\Omega)]$ and assertion {\em i)} of Theorem \ref{thm-Euler2} follows. 

 Now we shall prove assertion {\em ii)} of Theorem \ref{thm-Euler2}. First we have 
 $$ a_{\psi_0}(t) = \mbox{meas}\{x\in \Omega; \psi_0(x) <t \} =  \mbox{meas}(\Omega) - \mu_{0} (t), \quad \forall t\geq 0.$$
 We also have
 $$e_{min}= \mbox{ess inf} \ \psi_0,\quad \mbox{and} \quad  e_{max}= \max\{ e\in \RR: \mu_0(e) >0\} = \|\psi_0\|_{\infty}.$$
Therefore the pseudo-inverse $b_{\psi_0}$  of $a_{\psi_0}$ is given for all $\mu\in [0, \mbox{meas}(\Omega)[$ by 
\begin{align*} b_{\psi_0}(\mu) &= \sup\{ t<e_{max}:  a_{\psi_0}(t)  \leq \mu\}  \\
&=   \sup\{ t <e_{max}:    \mu_{0} (t) > \mbox{meas}(\Omega) - \mu\} \\
&=  \psi_0^{\sharp}\left( \mbox{meas}(\Omega) - \mu\right).
\end{align*}
In particular $|b_\sigma (\mu)| \leq  \|\psi_0\|_{L^\infty}$. Now, using this inequality and the obtained expression of $b_\sigma$,  assertion {\em ii)} is directly obtained by combining identity \fref{alhadra1} with inequality \fref{ineq11}.

 \end{document}